\documentclass[11pt]{article}
\usepackage{amsmath}
\usepackage{amssymb}
\usepackage{theorem}
\usepackage{pstricks}
\usepackage{euscript}
\usepackage{epic,eepic}
\PassOptionsToPackage{normalem}{ulem}
\definecolor{labelkey}{rgb}{0,0.08,0.45}
\definecolor{refkey}{rgb}{0,0.6,0.0}
\definecolor{Brown}{rgb}{0.45,0.0,0.05}
\definecolor{dgreen}{rgb}{0.00,0.49,0.00}
\definecolor{dblue}{rgb}{0,0.08,0.75}
\RequirePackage[dvips,colorlinks,hyperindex]{hyperref} 
\hypersetup{linktocpage=true,citecolor=dblue,linkcolor=dgreen}

\newcommand{\Argmind}[2]{\ensuremath{\underset{\substack{{#1}}}%
{\mathrm{Argmin}}\;\;#2 }}
\newcommand{\Argmax}[2]{\ensuremath{\underset{\substack{{#1}}}%
{\mathrm{Argmax}}\;\;#2 }}
 
\newcommand{\menge}[2]{\big\{{#1} \mid {#2}\big\}}

\newcommand{\emp}{\ensuremath{{\varnothing}}}

\newcommand{\infconv}{\ensuremath{\mbox{\small$\,\square\,$}}}
\newcommand{\scal}[2]{\left\langle{#1}\mid {#2} \right\rangle} 
\newcommand{\pscal}[2]{\boldsymbol{\langle}{#1}\mid{#2}\boldsymbol{\rangle}} 

\newcommand{\exi}{\ensuremath{\exists\,}}
\newcommand{\zeroun}{\ensuremath{\left]0,1\right[}}   
\newcommand{\HH}{\ensuremath{\mathcal H}}

\newcommand{\RR}{\ensuremath{\mathbb R}}

\newcommand{\RPP}{\ensuremath{\,\left]0,+\infty\right[}}
\newcommand{\RX}{\ensuremath{\,\left]-\infty,+\infty\right]}}
\newcommand{\NN}{\ensuremath{\mathbb N}}
\newcommand{\dom}{\ensuremath{\operatorname{dom}}}

\newcommand{\gr}{\ensuremath{\operatorname{gra}}}

\newcommand{\inte}{\ensuremath{\operatorname{int}}}
\newcommand{\cart}{\ensuremath{\mbox{\huge{$\times$}}}}

\newcommand{\Argmin}{\ensuremath{\operatorname{Argmin}}}
\newcommand{\ran}{\ensuremath{\operatorname{ran}}}
\newcommand{\zer}{\ensuremath{\operatorname{zer}}}

\newcommand{\Fix}{\ensuremath{\operatorname{Fix}}}
\newcommand{\Id}{\ensuremath{\operatorname{Id}}}

\newcommand{\weakly}{\ensuremath{\rightharpoonup}}

\newcommand{\pinf}{\ensuremath{+\infty}}
\newtheorem{theorem}{Theorem}[section]

\newtheorem{corollary}[theorem]{Corollary}
\newtheorem{proposition}[theorem]{Proposition}
\newtheorem{definition}[theorem]{Definition}
\theoremstyle{plain}{\theorembodyfont{\rmfamily}
}
\theoremstyle{plain}{\theorembodyfont{\rmfamily}
}
\theoremstyle{plain}{\theorembodyfont{\rmfamily}
}
\theoremstyle{plain}{\theorembodyfont{\rmfamily}
\newtheorem{example}[theorem]{Example}}
\theoremstyle{plain}{\theorembodyfont{\rmfamily}
\newtheorem{problem}[theorem]{Problem}}
\theoremstyle{plain}{\theorembodyfont{\rmfamily}
\newtheorem{remark}[theorem]{Remark}}
\theoremstyle{plain}{\theorembodyfont{\rmfamily}
}

\numberwithin{equation}{section}

\begin{document}

\title{\sffamily\huge Forward--partial inverse--forward splitting for
solving monotone inclusions 
\footnote{Contact author: 
L. M. Brice\~{n}o-Arias, {\ttfamily luis.briceno@usm.cl},
phone: +56 2 432 6662.
This work was supported by CONICYT under grants FONDECYT 3120054,
ECOS-CONICYT C13E03, Anillo ACT 1106, Math-Amsud N 13MATH01, and
by ``Programa de financiamiento basal'' from the Center for
Mathematical Modeling, Universidad de Chile.}}

\author{Luis M. Brice\~{n}o-Arias$^1$\\[5mm]
\small $\!^1$Universidad T\'ecnica Federico Santa Mar\'ia\\
\small Departamento de Matem\'atica\\
\small Santiago, Chile\\[4mm]
}

\date{}
\maketitle

\begin{abstract}
In this paper we provide a splitting method for finding a zero of the sum of a maximally monotone operator, a lipschitzian monotone operator, and a normal cone to a closed vectorial subspace of a real Hilbert space. The problem is characterized by a simpler monotone inclusion involving 
only two operators: the partial inverse of the maximally monotone operator with respect to the vectorial subspace and a suitable lipschitzian monotone operator. By applying the Tseng's method in this context we obtain a splitting algorithm that exploits the whole structure of the original problem and generalizes partial inverse and Tseng's methods. Connections with other methods available in the literature and applications to inclusions involving $m$ maximally monotone operators, to primal-dual composite monotone inclusions, and to zero-sum games are provided.

\end{abstract}

{\small 

}

\newpage

\section{Introduction}
\label{intro}
This paper is concerned with the numerical resolution of the problem of finding a zero of the sum of a set-valued maximally monotone operator $A$, a lipschitzian monotone operator $B$, and a normal cone $N_V$, where $V$ is a closed vectorial subspace of a real Hilbert space $\HH$. This problem arises in a wide range of areas such
as optimization \cite{Invp08,Spin85},
variational inequalities \cite{Lion67,Tsen90,Tsen91}, 
monotone operator theory
\cite{Ecks92,Lion79,Roc76a,Spin83}, partial differential
equations \cite{Lion67,Merc80,Zeid90},
economics \cite{Jofr07,Penn12}, signal and image processing
\cite{Aujo05,Cham97,Daub04},
evolution inclusions \cite{Aubi90,Hara81,Show97},
traffic theory
\cite{Bert82,Fuku96,Shef85}, and 
game theory \cite{Gilp12,Zink99}, among others.

In the particular case when the operator $B$ is zero, the problem is studied in \cite{Spin83} and it
is solved via the method of partial inverses.
On the other hand, when $V=\HH$, the normal cone is zero and our problem reduces to find a zero of $A+B$. In this case, the problem is studied in 
\cite{Tsen00}, where the forward-backward-forward splitting or Tseng's method is proposed for solving this problem 
(see also \cite{Siopt3} and the references therein). In addition, two methods are proposed in \cite{opti1} for finding a zero of $A+B+N_V$ in the particular case when $B$ is cocoercive.

In the general case, several
algorithms are available in the literature for finding a zero of $A+B+N_V$, but any of them exploits the intrinsic structure
of the problem. The forward-backward-forward splitting introduced in \cite{Tsen00} can be applied to the general case, but it needs to compute the resolvent of $A+N_V$, which is not always
easy to compute. It is preferable to activate $A$
and $N_V$ separately. Other ergodic approaches for solving the problem can be found in \cite{Bot13,Pass79}.
A disadvantage of these methods is the presence of vanishing parameters, which usually lead to numerical instabilities. The algorithms proposed in
\cite{Siopt3,Joca09,Spin83} permit to find a zero of the sum of finitely many maximally
monotone operators by activating them independently and without considering vanishing parameters. However, these methods involve implicit steps on $B$ by using its resolvent, which is not easy to compute in general. An algorithm proposed in \cite{Comb12} overcome this difficulty by activating explicitly the operator $B$. However, this method does not take advantage of
the vector subspace involved and, as a consequence, it needs to store additional auxiliary variables at each iteration, which can be difficult for high dimensional problems.

In this paper we propose a fully split method for finding a zero of $A+B+N_V$ by exploiting each of its intrinsic properties. 
The proposed algorithm computes, at each iteration, explicit steps on $B$ and the resolvent of the partial inverse of $A$ with respect to $V$ \cite{Spin83}, which can be explicitly found in several cases. In a particular instance, this resolvent becomes a Douglas-Rachford step \cite{Lion79,Svai10}, which activates separately $A$ and $N_V$. Hence, in this case our method can be perceived as a forward-Douglas-Rachford-forward splitting.
The proposed algorithm generalizes partial inverse 
and Tseng's methods in the particular instances when $B=0$ and $V=\HH$, respectively. We also provide connections with other methods in the literature and we illustrate the flexibility of this framework via some applications to inclusions involving $m$ maximally monotone operators, to primal-dual composite monotone inclusions, and to zero-sum games. In the application to primal-dual inclusions we introduce a new operation between set-valued operators, called {\em partial sum} with respect to a closed vectorial subspace, which preserves monotonicity and takes a central role in the problem and algorithm. On the other hand, in continuous zero-sum games, we provide an interesting splitting algorithm for calculating a Nash equilibrium that avoids the computation of the projection onto mixed strategy spaces in infinite dimensions by  performing simpler projections alternately.

The paper is organized as follows. In Section~\ref{sec:2} we provide
the notation and some preliminaries. We also obtain a relaxed version of Tseng's method \cite{Tsen00}, which is interesting in its own right. In
Section~\ref{sec:3} a characterization of Problem~\ref{prob:1} in terms of two appropriate monotone operators is given and a method for solving this problem is derived from the relaxed version of Tseng's algorithm. Moreover, we provide connections with other methods in the literature. Finally, in Section~\ref{sec:4} we apply our method to 
the problem of finding a zero of a sum of $m$ maximally
monotone operators and a lipschitzian monotone operator,
to a primal-dual composite monotone inclusion, and to continuous zero-sum games. The methods derived in each instance generalize and improve available algorithms in the literature in each context.
\section{Notation and Preliminaries}
\label{sec:2}
Throughout this paper, $\HH$ is a real Hilbert space with 
scalar product denoted by $\scal{\cdot}{\cdot}$ and  
associated norm $\|\cdot\|$. The symbols $\weakly$ and $\to$
denote, respectively, weak and strong convergence and $\Id$ denotes
the identity operator. The indicator function of a subset $C$ of $\HH$
is $\iota_C$, which takes the value $0$ in $C$ and $\pinf$ in $\HH\setminus C$.
If $C$ is non-empty, closed, and convex, the projection of $x$ onto
$C$, denoted by $P_Cx$, is the unique point in $\Argmin_{y\in
C}\|x-y\|$, and the normal cone to $C$ is the maximally monotone
operator
\begin{equation}
\label{e:normalcone}
N_C\colon\HH\to 2^{\HH}\colon x\mapsto
\begin{cases}
\menge{u\in\HH}{(\forall y\in C)\:\:\scal{y-x}{u}\leq0},\quad
&{\rm if\: }x\in C;\\
\emp,&{\rm otherwise}.
\end{cases}
\end{equation}
An operator $T\colon\HH\to\HH$ is $\beta$--cocoercive for
some $\beta\in\RPP$ if, for every $x\in\HH$ and $y\in\HH$,
$\scal{x-y}{Tx-Ty}\geq\beta\|Tx-Ty\|^2$, it is $\chi$-lipschitzian if, for every $x\in\HH$ and $y\in\HH$,
$\|Tx-Ty\|\leq\chi\|x-y\|$, it is non expansive if it is $1$-lipschitzian, 
and the set of fixed points of $T$ is given by
$\Fix T=\menge{x\in\HH}{Tx=x}$.

We denote by $\gr A=\menge{(x,u)\in\HH\times\HH}{u\in Ax}$ the graph
of a set-valued operator $A\colon\HH\to 2^{\HH}$, by
$\dom A=\menge{x\in\HH}{Ax\neq\emp}$ its domain, by
$\zer A=\menge{x\in\HH}{0\in Ax}$ its set of zeros, by $\ran A=\menge{u\in\HH}{(\exi x\in\HH)\:\:u\in Ax}$ its range, and by
$J_A=(\Id+A)^{-1}$ its resolvent. If $A$ is monotone, i.e.,
for every $(x,u)$ and $(y,v)$ in $\gr A$, 
$\scal{x-y}{u-v}\geq0$, then $J_A$ is 
single-valued and non expansive. In addition, if $\ran (\Id+A)=\HH$, $A$ is maximally monotone and $\dom J_A=\HH$. Let $A\colon\HH\to 2^{\HH}$ be
maximally monotone. The reflection operator of $A$ is $R_A=2J_A-\Id$,
which is non expansive. The partial inverse of $A$ with respect to
a vector subspace $V$ of $\HH$, denoted by $A_V$, is defined by 
\begin{equation}
\label{e:partialinv}
(\forall (x,y)\in\HH^2)\quad y\in A_Vx\quad\Leftrightarrow\quad 
(P_Vy+P_{V^{\bot}}x)\in A(P_Vx+P_{V^{\bot}}y).
\end{equation}
Note that $A_{\HH}=A$ and $A_{\{0\}}=A^{-1}$.
The following properties of the partial inverse will be useful throughout this paper.
\begin{proposition}
\label{p:propiedpi}
Let $A\colon\HH\to 2^{\HH}$ be a set-valued operator and let $V$ be a
vector subspace of $\HH$. Then the following hold.
\begin{enumerate}
 \item\label{p:propiedpii} $(A_V)^{-1}=(A^{-1})_V=A_{V^{\bot}}$.
 \item\label{p:propiedpiii} $P_V(A+N_V)^{-1}P_V=P_V(A_{V^{\bot}}+N_V)P_V$.
\end{enumerate}
\end{proposition}
\begin{proof}
\ref{p:propiedpii}: Let $(x,u)\in\HH^2$. We have from \eqref{e:partialinv} that
\begin{align}
u\in (A_V)^{-1}x\quad&\Leftrightarrow\quad
x\in A_Vu\nonumber\\
&\label{e:auxpi}\Leftrightarrow\quad
P_Vx+P_{V^{\bot}}u\in A(P_Vu+P_{V^{\bot}}x)\\
&\Leftrightarrow\quad
P_Vu+P_{V^{\bot}}x\in A^{-1}(P_Vx+P_{V^{\bot}}u)\nonumber\\
&\Leftrightarrow\quad
x\in (A^{-1})_Vu.
\end{align}
On the other hand, it follows from \eqref{e:auxpi} and \eqref{e:partialinv} that $u\in (A_V)^{-1}x$ is equivalent to $u\in A_{V^{\bot}}x$.
\ref{p:propiedpiii}: Let $(x,u)\in\HH^2$. We deduce from \ref{p:propiedpii} and \eqref{e:partialinv} that
\begin{align}
u\in P_V(A+N_V)^{-1}(P_Vx) \quad 
&\Leftrightarrow\quad (u\in V)\quad P_Vx\in Au+N_Vu\nonumber\\
&\Leftrightarrow\quad(u\in V)(\exi y\in V^{\bot})\quad P_Vx-y\in Au\nonumber\\ 
&\Leftrightarrow\quad(u\in V)(\exi y\in V^{\bot})\quad u-y\in A_{V^{\bot}}(P_Vx)\nonumber\\ 
&\Leftrightarrow\quad u\in P_V(A_{V^{\bot}}+N_V)(P_Vx), 
\end{align} 
which yields the result.
\end{proof}

The following result is a relaxed version of the method originally proposed in \cite{Tsen00} over some modifications developed in \cite{Nfao1,Siopt3}. 
\begin{proposition}
\label{p:tsengrel}
Let $\mathcal{A}\colon\HH\to 2^{\HH}$ be maximally monotone and let $\mathcal{B}\colon\HH\to\HH$ be monotone and $\eta$--lipschitzian such that $\zer(\mathcal{A}+\mathcal{B})\neq\varnothing$. Moreover, let $z_0\in\HH$, let $\varepsilon\in\left]0,\max\{1,1/2\eta\}\right[$,
let $(\delta_n)_{n\in\NN}$ be a sequence in
$\left[\varepsilon,(1/\eta)-\varepsilon\right]$, let
$(\lambda_n)_{n\in\NN}$ be a sequence in $\left[\varepsilon,1\right]$, and iterate, for every $n\in\NN$,
\begin{align}
\label{e:tsengrel}
&\left 
\lfloor 
\begin{array}{l}
r_n=z_n-\delta_n\mathcal{B}z_n\\
s_n=J_{\delta_n\mathcal{A}}r_n\\
t_n=s_n-\delta_n\mathcal{B}s_n\\
z_{n+1}=z_n+\lambda_n(t_n-r_n).
\end{array}
\right.
\end{align}
Then, $z_n\weakly\bar{z}$ for some $\bar{z}\in\zer(\mathcal{A}+\mathcal{B})$ and $z_{n+1}-z_n\to0$.
\end{proposition} 
\begin{proof}
First note that \eqref{e:tsengrel} yields
\begin{equation}
\label{e:Asn}
(\forall n\in\NN)\quad \delta_n^{-1}(r_n-s_n)\in\mathcal{A}s_n.
\end{equation}
Let $z\in\zer(\mathcal{A}+\mathcal{B})$ and fix $n\in\NN$. It follows from \cite[Lemma~3.1]{Tsen00} and \eqref{e:tsengrel} that
\begin{align}
\label{e:aux1}
\|z_{n+1}-z\|^2&=\|(1-\lambda_n)(z_n-z)+\lambda_n\big(s_n-\delta_n(\mathcal{B}s_n-\mathcal{B}z_n)-z\big)\|^2\nonumber\\
&=(1-\lambda_n)\|z_n-z\|^2+\lambda_n\|s_n-\delta_n(\mathcal{B}s_n-\mathcal{B}z_n)-z\|^2-\lambda_n(1-\lambda_n)\|t_n-r_n\|^2\nonumber\\
&\leq (1-\lambda_n)\|z_n-z\|^2+\lambda_n\big(\|z_n-z\|^2+\delta_n^2\|\mathcal{B}s_n-\mathcal{B}z_n\|^2-\|s_n-z_n\|^2\big)\nonumber\\
&\hspace{8.85cm}-\lambda_n(1-\lambda_n)\|t_n-r_n\|^2\nonumber\\
&\leq \|z_n-z\|^2-\big(1-(\delta_n\eta)^2\big)\|s_n-z_n\|^2-\lambda_n(1-\lambda_n)\|t_n-r_n\|^2.
\end{align}
Hence, since $\delta_n<1/\eta$ and $0<\lambda_n\leq1$, we obtain $\|z_{n+1}-z\|^2\leq\|z_n-z\|^2$, which yields the boundedness of the sequence $(z_k)_{k\in\NN}$. Moreover, we deduce from \eqref{e:aux1} and \cite[Lemma~2.1]{Siopt3} that $(\|s_k-z_k\|^2)_{k\in\NN}$ and $(\|t_k-r_k\|^2)_{k\in\NN}$ are summable and, in particular,
\begin{equation}
\label{e:tozero}
s_k-z_k\to 0\quad\text{and}\quad t_k-r_k\to 0,
\end{equation}
which yields $z_{k+1}-z_k=\lambda_k(t_k-r_k)\to0$.
By setting, for every $k\in\NN$, $u_k=\delta_k^{-1}(r_k-t_k)$,
it follows from \eqref{e:tsengrel}, \eqref{e:Asn}, and \eqref{e:tozero} that 
\begin{equation}
\label{e:everyt}
(\forall k\in\NN)\quad 0\leftarrow u_k=\delta_k^{-1}(r_k-s_k)+\mathcal{B}s_k\in(\mathcal{A}+\mathcal{B})s_k.
\end{equation}
Now let us take $w\in\HH$ be any sequential weak cluster point of $(z_k)_{k\in\NN}$, say $z_{k_\ell}\weakly w$.
Then, it follows from \eqref{e:tozero} and \eqref{e:everyt} that 
\begin{equation}
s_{k_\ell}\weakly w,\:\:u_{k_\ell}\to 0, \quad\text{and}\quad (s_{k_\ell},u_{k_\ell})\in\gr(\mathcal{A}+\mathcal{B}). 
\end{equation}
Since $\mathcal{B}$ is monotone and continuous, it is maximally monotone \cite[Corollary~20.25]{Livre1}. Moreover, since $\dom\mathcal{B}=\HH$, we deduce from \cite[Corollary~24.4(i)]{Livre1} that $\mathcal{A}+\mathcal{B}$ is maximally monotone and, hence, its graph is sequentially closed in $\HH^{\rm weak}\times\HH^{\rm strong}$ \cite[Proposition~20.33(ii)]{Livre1}. Therefore, we conclude from \eqref{e:everyt} that $w\in\zer(\mathcal{A}+\mathcal{B})$ and from 
\cite[Lemma~2.2]{Siopt3} we deduce that there exists $\bar{z}\in\zer(\mathcal{A}+\mathcal{B})$ such that $z_n\weakly\bar{z}$ which yields the result. 
\end{proof}

\begin{remark}
As in \cite[Theorem~2.5]{Siopt3}, absolutely summable errors can be incorporated in each step of the algorithm in \eqref{e:tsengrel}. However, for ease of presentation throughout the document, we only provide the error free version.
\end{remark}

For complements and further background 
on monotone operator theory and algorithms, the reader is referred to
\cite{Aubi90,Livre1,Spin83,Zeid90}.

\section{Forward--Partial Inverse--Forward Splitting}
\label{sec:3}
We aim at solving the following problem.
\begin{problem}
\label{prob:1}
Let $\HH$ be a real Hilbert space and let $V$ be a closed vector subspace of $\HH$. Let $A\colon\HH\to 2^{\HH}$ be a maximally monotone operator and let $B\colon \HH\to\HH$ be a monotone and $\chi$--lipschitzian operator.
The problem is to
\begin{equation}
\label{e:mean0}
\text{find}\;\;x\in\HH\quad\text{such that}\quad 
0\in Ax+Bx+N_Vx,
\end{equation}
under the assumption $\zer(A+B+N_V)\neq\varnothing$. 
\end{problem}

In this section we provide our method for solving
Problem~\ref{prob:1}. We first provide a characterization of the solutions to Problem~\ref{prob:1}, which motivates our algorithm. Its convergence to a solution to Problem~\ref{prob:1} is then proved.

\subsection{Characterization}
The following result provides a characterization of the solutions to Problem~\ref{prob:1} in terms of two suitable monotone operators.
\begin{proposition}
\label{p:1}
Let $\gamma\in\RPP$ and $\HH$, $A$, $B$, and $V$ be as in
Problem~\ref{prob:1}.
Define 
\begin{equation}
\label{e:defmaxmon2}
\begin{cases}
\mathcal{A}_{\gamma}=(\gamma A)_V\colon\HH\to 2^{\HH}\\
\mathcal{B}_{\gamma}=\gamma P_V\circ B\circ P_V\colon\HH\to V.
\end{cases} 
\end{equation}
Then the following hold.
\begin{enumerate}
\item\label{p:1i} $\mathcal{A}_{\gamma}$ is maximally monotone and,
for every $\delta\in\RPP$ and $x\in\HH$,
$J_{\delta\mathcal{A}_{\gamma}}x=P_Vp+\gamma P_{V^{\bot}}q$, where
$p$ and $q$ in $\HH$ are such that $x=p+\gamma q$ and 
\begin{equation}
\label{e:Spininc}
\frac{P_Vq}{\delta}+P_{V^{\bot}}q\in
A\left(P_Vp+\frac{P_{V^{\bot}}p}{\delta}\right).
\end{equation}
In particular, for every $x\in\HH$, 
$J_{\mathcal{A}_{\gamma}}x=2P_VJ_{\gamma A}-J_{\gamma
A}+\Id-P_V=(\Id+R_{N_V}R_{\gamma A})/2$.
\item\label{p:1ii} $\mathcal{B}_{\gamma}$ is
$\gamma\chi$--lipschitzian and monotone.
\item\label{p:1iii} Let $x\in\HH$. Then $x$ is a solution to
Problem~\ref{prob:1}
if and only if $x\in V$ and
\begin{equation}
\hspace{-.2cm}\big(\exi y\in V^{\bot}\cap(Ax+Bx)\big)\:\:\text{such that}\:\:\:
x+\gamma
(y-P_{V^{\bot}}Bx)\in\zer(\mathcal{A}_{\gamma}+\mathcal{B}_{\gamma}). 
\end{equation}
\end{enumerate}

\end{proposition}
\begin{proof}
\ref{p:1i}: Since $\gamma A$ is maximally monotone, $\mathcal{A}_{\gamma}$ inherits this property \cite[Proposition~2.1]{Spin83}. In addition, it follows from
\eqref{e:partialinv} that, for every $(p,q,x)\in\HH^3$ such that
$p+\gamma q=x$ and every $\delta\in\RPP$,
\begin{align}
\frac{P_Vq}{\delta}+P_{V^{\bot}}q\in
A\Big(P_Vp+\frac{P_{V^{\bot}}p}{\delta}
\Big)\:
&\Leftrightarrow\:\: 
\frac{\gamma P_Vq}{\delta}+\gamma P_{V^{\bot}}q\in
\gamma A\left(P_Vp+\frac{P_{V^{\bot}}p}{\delta}
\right)\nonumber\\
&\Leftrightarrow\:\: 
\frac{\gamma P_Vq}{\delta}+\frac{P_{V^{\bot}}p}{\delta}\in
\mathcal{A}_{\gamma}\left(P_Vp+\gamma P_{V^{\bot}}q
\right)\nonumber\\
&\Leftrightarrow\:\:
\gamma P_Vq+P_{V^{\bot}}p\in
\delta\mathcal{A}_{\gamma}\left(P_Vp+\gamma P_{V^{\bot}}q
\right)\nonumber\\
&\Leftrightarrow\:\:
P_Vp+\gamma P_{V^{\bot}}q=J_{\delta\mathcal{A}_{\gamma}}(p+\gamma
q)\nonumber\\
&\Leftrightarrow\:\:
P_Vp+\gamma P_{V^{\bot}}q=J_{\delta\mathcal{A}_{\gamma}}x.
\end{align}
In particular, if $\delta=1$, \eqref{e:Spininc} reduces to
$p=J_{\gamma A}(p+\gamma q)=J_{\gamma A}x$ and, hence,
\begin{align}
J_{\mathcal{A}_{\gamma}}x&=P_V(J_{\gamma A}x)+P_{V^{\bot}}
(x-J_{\gamma A}x)\nonumber\\
&=2P_VJ_{\gamma A}x-J_{\gamma A}x+x-P_Vx\nonumber\\
&=\frac12\left(x+2P_V(2J_{\gamma A}x-x)-2J_{\gamma
A}x+x\right)\nonumber\\
&=\frac12\left(x+R_{N_V}R_{\gamma A}x\right).
\end{align}
\ref{p:1ii}: Let $(x,y)\in\HH^2$. We have from
\eqref{e:defmaxmon2}, the monotonicity of $B$, the fact that $P_V$ is linear,
and $P_V^*=P_V$ that
$\scal{x-y}{\mathcal{B}_{\gamma}x-\mathcal{B}_{\gamma}y}=\gamma\scal{
P_Vx-P_V y}{B(P_Vx)-B(P_Vy)}\geq 0$,
and, from the lipschitzian property on $B$ and \eqref{e:defmaxmon2} we obtain
$
\|\mathcal{B}_{\gamma}x-\mathcal{B}_{\gamma}
y\|\leq
\gamma\|B(P_Vx)-B(P_Vy)\|
\leq\gamma\chi\|P_Vx-P_Vy\|\leq\gamma\chi\|x-
y\|$.
\ref{p:1iii}:
Let $x\in\HH$ be a solution to Problem~\ref{prob:1}. 
We have $x\in V$ and there exists $y\in V^{\bot}=N_Vx$ such that 
$y\in Ax+Bx$. 
Since $B$ is single valued and $P_V$ is linear, it follows from
\eqref{e:partialinv} that
\begin{align}
\label{e:carac1}
y\in Ax+Bx\quad
&\Leftrightarrow\quad
\gamma y-\gamma Bx\in \gamma Ax\nonumber\\
&\Leftrightarrow\quad
-\gamma P_V(Bx)
\in (\gamma A)_V\big(x+\gamma (y-P_{V^{\bot}}Bx)\big)\nonumber\\
&\Leftrightarrow\quad
0\in (\gamma A)_V(x+\gamma (y-P_{V^{\bot}}Bx))+\gamma
P_V\big(B\big(P_V(x+\gamma (y-P_{V^{\bot}}Bx))\big)\big)\nonumber\\
&\Leftrightarrow\quad
x+\gamma (y-P_{V^{\bot}}Bx)\in\zer
(\mathcal{A}_{\gamma}+\mathcal{B}_{\gamma}),
\end{align}
which yields the result.
\end{proof}
\begin{remark}
Note that the characterization in Proposition~\ref{p:1}\ref{p:1iii}
yields
$
Z=P_V\big(\zer(\mathcal{A}_{\gamma}+\mathcal{B}_{\gamma})\big).
$
\end{remark}

\subsection{Algorithm and convergence}

In the following result we propose our algorithm and we prove its convergence to a solution to Problem~\ref{prob:1}. Since Proposition~\ref{p:1} asserts that Problem~\ref{prob:1} can be solved via a monotone inclusion involving a maximally monotone operator and
a single-valued lipschitzian monotone operator, our method
is obtained as a consequence of Proposition~\ref{p:tsengrel}, which is inspired from \cite{Siopt3,Tsen00}. 

\begin{theorem}
\label{t:1}
Let $\HH$, $V$, $A$, and $B$, be as in Problem~\ref{prob:1},
let $\gamma\in\RPP$, let
$\varepsilon\in\big]0,\max\{1,\frac{1}{2\gamma\chi}\}\big[$,
let $(\delta_n)_{n\in\NN}$ be a sequence in
$\big[\varepsilon,\frac{1}{\gamma\chi}-\varepsilon\big]$, and let
$(\lambda_n)_{n\in\NN}$ be a sequence in $\left[\varepsilon,1\right]$.
Moreover, 
let $x_0\in V$, let $y_0\in V^{\bot}$, and iterate, for every $n\in\NN$, 
\begin{align}
\label{e:auxprop}
1.\:&
\text{\rm find }(p_n,q_n)\in\HH^2\:\text{\rm such that }
x_n-\delta_n\gamma P_VBx_n+\gamma y_n=p_n+\gamma
q_n\nonumber\\
&\text{\rm and } \frac{P_Vq_n}{\delta_n}+P_{V^{\bot}}q_n\in
A\Big(P_Vp_n+\frac{P_{V^{\bot}}p_n}{\delta_n}\Big).\\
2.\:&\text{\rm set }
x_{n+1}
=x_n+\lambda_n(P_Vp_n+\delta_n\gamma
P_V(Bx_n-BP_Vp_n)-x_n)\:\:\nonumber\\
&\text{\rm and }
y_{n+1}=y_n+\lambda_n(P_{V^{\bot}}q_n-y_n).\: \text{\rm Go to
}1.\nonumber
\end{align}
Then, the sequences $(x_n)_{n\in\NN}$ and $(y_n)_{n\in\NN}$ are in 
$V$ and $V^{\bot}$, respectively, $x_n\weakly \overline{x}$ and $y_n\weakly
\overline{y}$ for some solution $\overline{x}\in\zer(A+B+N_V)$ and $\overline{y}\in
V^{\bot}\cap(A\overline{x}+P_VB\overline{x})$, $x_{n+1}-x_n\to 0\,$, and $\,y_{n+1}-y_n\to 0$.
\end{theorem}
\begin{proof}
Since $x_0\in V$ and $y_0\in V^{\bot}$, \eqref{e:auxprop}
yields $(x_n)_{n\in\NN}\subset V$ and $(y_n)_{n\in\NN}\subset
V^{\bot}$.
Thus, for every $n\in\NN$, it follows from \eqref{e:auxprop} and
Proposition~\ref{p:1}\ref{p:1i} that
\begin{equation}
\label{e:aux12}
P_Vp_n+\gamma P_{V^{\bot}}q_n=J_{
\delta_n(\gamma A)_V}(x_n+\gamma
y_n-\delta_n\gamma P_VBx_n).
\end{equation}
For every $n\in\NN$, denote by $z_n=x_n+\gamma y_n$ and by
\begin{equation}
s_n=J_{\delta_n(\gamma A)_V}(x_n+\gamma
y_n-\delta_n\gamma P_VBx_n)=J_{\delta_n(\gamma
A)_V}(z_n-\delta_n\gamma
P_VBP_Vz_n)
=J_{\delta_n\mathcal{A}_{\gamma}}(z_n-\delta_n\mathcal{B}_{
\gamma}z_n).
\end{equation}
Hence, it follows from \eqref{e:aux12} that $P_Vp_n=P_Vs_n$, $\gamma P_{V^{\bot}}q_n=P_{V^{\bot}}s_n$,
and, from \eqref{e:auxprop}, we obtain
\begin{equation}
\begin{cases}
x_{n+1}=x_n+\lambda_n (P_Vs_n+\delta_n\gamma P_V(Bx_n-BP_Vs_n)-x_n)\\
\gamma y_{n+1}=\gamma y_n+\lambda_n (
P_{V^{\bot}}s_n-\gamma y_n).
\end{cases}
\end{equation}
By adding the latter equations we deduce that
the algorithm described in \eqref{e:auxprop} can be written
equivalently as
\begin{align}
\label{e:fbfred}
(\forall n\in\NN)\quad 
&\left 
\lfloor 
\begin{array}{l}
r_n=z_n-\delta_n\mathcal{B}_{\gamma}z_n\\
s_n=J_{\delta_n\mathcal{A}_{\gamma}}r_n\\
t_n=s_n-\delta_n\mathcal{B}_{\gamma}s_n\\
z_{n+1}=z_n+ \lambda_n
(t_n-r_n),
\end{array}
\right. 
\end{align}
which is a particular instance of \eqref{e:tsengrel} when $\mathcal{B}=\mathcal{B}_{\gamma}$ and $\mathcal{A}=\mathcal{A}_{\gamma}$.
Therefore, it follows from
Proposition~\ref{p:1}\ref{p:1i}\&\ref{p:1ii} and Proposition~\ref{p:tsengrel} that $z_n\weakly \overline{z}\in\zer(\mathcal{A}_{\gamma}+\mathcal{B}_{\gamma})$ and $z_{n+1}-z_n\to0$. By defining  $\overline{x}:=P_V\overline{z}\in Z$ and
$\overline{y}:=P_{V^{\bot}}\overline{z}/\gamma\in
(A\overline{x}+B\overline{x})-P_{V^{\bot}}B\overline{x}=A\overline{x}
+P_VB\overline{x}$, the results follow from Proposition~\ref{p:1}\ref{p:1iii} and Proposition~\ref{p:tsengrel}.
\end{proof}
\newpage

\begin{remark}\
\begin{enumerate}
\item 
It is known that the forward--backward--forward
splitting admits errors in the computations of the
operators involved \cite{Nfao1,Siopt3}. In our algorithm these
inexactitudes
have not been considered for simplicity.
\item 
In the particular case when $\lambda_n\equiv1$ and $B\equiv0$
($\chi=0$), \eqref{e:auxprop} reduces to the classical partial
inverse method proposed in \cite{Spin83} for finding 
$x\in V$ such that there exists $y\in V^{\bot}$ satisfying $y\in Ax$.
\item As in \cite{Siopt3}, under further assumptions on the operators $\mathcal{A}_{\gamma}$ and/or $\mathcal{B}_{\gamma}$, e.g., as demiregularity (see \cite[Definition~2.3\&Proposition~2.4]{Sicon1}), strong convergence can be achieved.
\end{enumerate}
\end{remark}

The sequence $(\delta_n)_{n\in\NN}$ in Theorem~\ref{t:1} can be
manipulated 
in order to accelerate the algorithm. However, as in \cite{Spin83}, 
{\em Step} 1 in Theorem~\ref{t:1} is not always easy to compute. 
The following result show us a particular case of our 
method in which {\em Step} 1 can be obtained explicitly 
when the resolvent of $A$ is computable. The method can be seen as a
forward-Douglas-Rachford-forward splitting for solving
Problem~\ref{p:1}.

\begin{corollary}
\label{c:0}
Let $\HH$, $V$, $A$, and $B$, be as in Problem~\ref{prob:1},
let $\gamma\in\left]0,1/\chi\right[$, let
$\varepsilon\in\left]0,1\right[$, and let
$(\lambda_n)_{n\in\NN}$ be a sequence in $\left[\varepsilon,1\right]$. Moreover,
let $z_0\in\HH$, and iterate, for every $n\in\NN$,
\begin{align}
\label{e:algo}
&\left 
\lfloor 
\begin{array}{l}
r_n=z_n-\gamma P_V BP_Vz_n\\
p_n=J_{\gamma A}r_n\\
s_n=2P_Vp_n-p_n+r_n-P_Vr_n\\
t_n=s_n-\gamma P_V BP_Vs_n\\
z_{n+1}=z_n+\lambda_n (t_n-r_n).
\end{array}
\right.
\end{align}
Then, by setting, for every $n\in\NN$, $x_n=P_Vz_n$ and $y_n=P_{V^{\bot}}z_n/\gamma$, we have $x_n\weakly\bar{x}$ and $y_n\weakly\bar{y}$ for some $\overline{x}\in\zer(A+B+N_V)$ and $\overline{y}\in
V^{\bot}\cap(A\overline{x}+P_VB\overline{x})$, $x_{n+1}-x_n\to 0$, and $y_{n+1}-y_n\to 0$. 
\end{corollary}
\begin{proof}
Indeed, it follows from the proof of Theorem~\ref{t:1} that
\eqref{e:auxprop} is equivalent to \eqref{e:fbfred}, where, for every $n\in\NN$, $z_n=x_n+\gamma y_n$. In the particular case when
$\delta_n\equiv1\in\left]0,1/(\gamma\chi)\right[$, it follows from 
Proposition~\ref{p:1}\ref{p:1i} that \eqref{e:fbfred} reduces to
\eqref{e:algo}. Hence, the results follow from Theorem~\ref{t:1}.
\end{proof}

\begin{remark}\
\begin{enumerate}
 \item 
Note that, when $V=\HH$ and $\lambda_n\equiv1$, we have
$V^{\bot}=\{0\}$, $P_V=\Id$, $(\Id+R_{N_V}R_{\gamma A})/2=J_{\gamma
A}$, and, therefore, \eqref{e:algo} reduces
to
\begin{equation}
(\forall n\in\NN)\quad
\left 
\lfloor 
\begin{array}{l}
r_n=x_n-\gamma Bx_n\\
s_n=J_{\gamma A}r_n\\
t_n=s_n-\gamma Bs_n\\
x_{n+1}=x_n+t_n-r_n,
\end{array}
\right.
\end{equation}
which is a version with constant step size of the modified forward-backward splitting \cite{Tsen00} for finding a zero of $A+B$. 
\item On the other hand, when $B\equiv 0$,  \eqref{e:algo} reduces to 
\begin{equation}
(\forall n\in\NN)\quad
\left 
\lfloor 
\begin{array}{l}
s_n=(z_n+R_{N_V}R_{\gamma A}z_n)/2\\
z_{n+1}=z_n+\lambda_n(s_n-z_n),
\end{array}
\right.
\end{equation}
which is the Douglas-Rachford splitting method \cite{Lion79,Svai10} for finding
$x\in\HH$ such that $x\in N_Vx+Ax$. It coincides with Spingarn's
partial inverse method with constant step size.
\end{enumerate}

\end{remark}
\section{Applications}
\label{sec:4}
In this section we study three applications of our algorithm. We first apply Theorem~\ref{t:1} to the problem of finding a zero of the sum of $m$ maximally
monotone operators and a monotone lipschitzian operator. Secondly, we study a primal-dual composite monotone inclusion
involving normal cones and we obtain from Theorem~\ref{t:1}
a primal-dual method for solving this problem. Finally, we study the application of our method in the framework of continuous zero-sum games.
Connections with other methods in
each framework are also provided. 
\subsection{Inclusion Involving the Sum of $m$ Monotone Operators}
\label{ssec:41}
Let us consider the following problem.
\begin{problem}
\label{prob:Fad}
Let $(\mathsf{H},|\cdot|)$ be a real Hilbert space, 
for every $i\in\{1,\ldots,m\}$, 
let $\mathsf{A}_i\colon\mathsf{H}\to 2^{\mathsf{H}}$
be a maximally monotone operator, and
let $\mathsf{B}\colon\mathsf{H}\to\mathsf{H}$ be a
monotone and $\chi$--lipschitzian operator.
The problem is to
\begin{equation}
\text{find}\quad \mathsf{x}\in\mathsf{H}\quad\text{such that}
\quad \mathsf{0}\in\sum_{i=1}^m\mathsf{A}_i\mathsf{x}
+\mathsf{B}\mathsf{x},
\end{equation}
under the assumption that solutions exist.
\end{problem}

Problem~\ref{prob:Fad} has several applications in image processing, 
principally in the variational setting (see, e.g.,
\cite{Invp08,Fadi12} and the references therein), variational
inequalities \cite{Tsen90,Tsen91}, partial differential
equations \cite{Merc80,Zeid90}, and economics \cite{Jofr07,Penn12}, among
others. In \cite{Fadi12,Bang12}, Problem~\ref{prob:Fad} is solved by a fully split algorithm in the particular case when $\mathsf{B}$ is cocoercive. Nevertheless, this approach does not seem to work in the general case. In \cite{Comb12} a method for solving a more general problem than Problem~\ref{prob:Fad} is proposed. However, this approach stores and updates at each iteration $m$ dual variables in order to solve \eqref{prob:Fad} and its dual simultaneously. This generality does not allow to exploit the intrinsic properties of
Problem~\ref{prob:Fad}, which may be unfavourable in large scale systems. Our method is obtained as a
consequence of Theorem~\ref{t:1} for a suitable closed vectorial subspace and exploits the whole structure of the problem.

Let us first provide a connection between
Problem~\ref{prob:Fad} and Problem~\ref{prob:1} via product space techniques.
Let $(\omega_i)_{1\leq i\leq m}$ be real numbers in $\zeroun$ such that
$\sum_{i=1}^m\omega_i=1$, let $\HH$ be the real 
Hilbert space obtained by endowing 
the Cartesian product $\mathsf{H}^m$ with the scalar product and
associated norm respectively defined by 
\begin{equation}
\label{e:prodscal}
\scal{\cdot}{\cdot}\colon (x,y)\mapsto
\sum_{i=1}^m\omega_i\pscal{\mathsf{x}_i}{\mathsf{y}_i}\quad\text{and}
\quad
\|\cdot\|\colon
x\mapsto\sqrt{\sum_{i=1}^m\omega_i\mathsf{|}\mathsf{x}_i\mathsf{|}^2},
\end{equation}
where $x=(\mathsf{x}_i)_{1\leq i\leq m}$ is a generic element of $\HH$.
\begin{proposition}
\label{p:app1}
Let $\mathsf{H}$, $(\mathsf{A}_i)_{1\leq i\leq m}$, and $\mathsf{B}$ be as in 
Problem~\ref{prob:Fad}, and define
\begin{equation}
\label{e:defprodc}
\begin{cases}
V=\menge{x=(\mathsf{x}_i)_{1\leq i\leq m}\in\HH}{\mathsf{x}_1=\cdots=\mathsf{x}_m}\\
j\colon\mathsf{H}\to V\subset\HH\colon\mathsf{x}\mapsto(\mathsf{x},\ldots,\mathsf{x})\\
A\colon\HH\to 2^{\HH}\colon x\mapsto \frac{1}{\omega_1}\mathsf{A}_1\mathsf{x}_1\times\cdots\times
\frac{1}{\omega_m}\mathsf{A}_m\mathsf{x}_m\\
B\colon\HH\to\HH\colon x\mapsto(\mathsf{B}\mathsf{x}_1,\ldots,\mathsf{B}\mathsf{x}_m).
\end{cases}
\end{equation}
Then the following hold.
\begin{enumerate}
\item\label{p:app1i} $V$ is a closed vector subspace of $\HH$, 
$P_V\colon (\mathsf{x}_i)_{1\leq
i\leq m}\mapsto j(\sum_{i=1}^m\omega_i\mathsf{x}_i)$, and 
\begin{equation}
N_V\colon x\mapsto
\begin{cases}
V^{\bot}=\menge{x=(\mathsf{x}_i)_{1\leq i\leq m}\in\HH}{\sum_{i=1}^m\omega_i\mathsf{x}_i=\mathsf{0}},
\quad&\text{if}\:\:x\in V;\\
\emp,&\text{otherwise}.
\end{cases}
\end{equation}
\item\label{p:app1ii} $j\colon\mathsf{H}\to V$ is a bijective isometry and $j^{-1}\colon(\mathsf{x},\ldots,\mathsf{x})
\mapsto \mathsf{x}$.
\item\label{p:app1iii} $A$ is a maximally monotone operator and, for
every $\gamma\in\RPP$, 
$J_{\gamma A}\colon
(\mathsf{x}_i)_{1\leq
i\leq m}\mapsto(J_{\gamma\mathsf{A}_i/\omega_i}\mathsf{x}_i)$.
\item\label{p:app1iv} $B$ is monotone and $\chi$--lipschitzian, 
$B(j(\mathsf{x}))=j(\mathsf{B}\mathsf{x})$, and $B(V)\subset V$.
\item\label{p:app1v} For every $\mathsf{x}\in\mathsf{H}$, 
$\mathsf{x}$ is a solution to Problem~\ref{prob:Fad}
if and only if $j(\mathsf{x})\in\zer(A+B+N_V)$.
\end{enumerate}
\end{proposition}
\begin{proof}
\ref{p:app1i}\&\ref{p:app1ii}: They follow from \eqref{e:normalcone}
and easy computations.
\ref{p:app1iii}: See \cite[Proposition~23.16]{Livre1}.
\ref{p:app1iv}: They follow from straightforward computations by using \eqref{e:defprodc}, \eqref{e:prodscal}, and the properties on $\mathsf{B}$.  
\ref{p:app1v}: Let $\mathsf{x}\in\mathsf{H}$. We have
\begin{align}
\mathsf{0}\in\sum_{i=1}^m\mathsf{A}_i\mathsf{x}+\mathsf{B}\mathsf{x}
\quad
&\Leftrightarrow\quad\bigg(\exi(\mathsf{y}_i)_{1\leq i\leq m}\in
\overset{m}{\underset{i=1}{\cart}}\mathsf{A}_i\mathsf{x}\bigg)
\quad
\mathsf{0}=\sum_{i=1}^m\mathsf{y}_i+\mathsf{B}\mathsf{x}\nonumber\\
&\Leftrightarrow\quad\bigg(\exi(\mathsf{y}_i)_{1\leq i\leq m}\in
\overset{m}{\underset{i=1}{\cart}}\mathsf{A}_i\mathsf{x}\bigg)
\quad 
\mathsf{0}=\sum_{i=1}^m\omega_i(-\mathsf{y}_i/\omega_i-\mathsf{B}
\mathsf{x})\nonumber\\
&\Leftrightarrow\quad\bigg(\exi(\mathsf{y}_i)_{1\leq i\leq m}\in
\overset{m}{\underset{i=1}{\cart}}\mathsf{A}_i\mathsf{x}\bigg)-(\mathsf{y}_1/\omega_1,\ldots,\mathsf{y}_m/\omega_m)-j(\mathsf{B}
\mathsf{x})\in
V^{\bot}\nonumber\\
&\Leftrightarrow\quad
0\in A(j(\mathsf{x}))+B(j(\mathsf{x}))+N_V(j(\mathsf{x}))\nonumber\\
&\Leftrightarrow\quad j(\mathsf{x})\in\zer(A+B+N_V),
\end{align}
which yields the result.
\end{proof}

The following result provides a method for solving
Problem~\ref{prob:Fad}. It is a direct consequence of
Corollary~\ref{c:0} applied to the equivalent monotone inclusion in Proposition~\ref{p:app1}\ref{p:app1v}. 
\begin{theorem}
\label{t:4}
Let $\mathsf{H}$, $(\mathsf{A}_i)_{1\leq i\leq m}$, and $\mathsf{B}$ be as in 
Problem~\ref{prob:Fad}, let $\gamma\in\left]0,1/\chi\right[$, let
$\varepsilon\in\left]0,1\right[$, and let $(\lambda_n)_{n\in\NN}$ be
a sequence in $\left[\varepsilon,1\right]$. Moreover, let
$(\mathsf{z}_{i,0})_{1\leq i\leq
m}\in\mathsf{H}^m$
and iterate, for every $n\in\NN$,
\begin{align}
\label{e:algo2}
&\left\lfloor 
\begin{array}{l}
\mathsf{x}_{n}=\sum_{j=1}^m\omega_j\mathsf{z}_{j,n}\\
\text{\rm For }i=1,\ldots,m\\
\left\lfloor
\begin{array}{l}
\mathsf{r}_{i,n}=\mathsf{z}_{i,n}-\gamma
\mathsf{B}\mathsf{x}_n\\
\mathsf{p}_{i,n}=J_{\gamma
\mathsf{A}_i/\omega_i}\mathsf{r}_{i,n}\\
\end{array}
\right.\\
\mathsf{q}_{n}=\sum_{j=1}^m\omega_j\mathsf{p}_{j,n}\\
\text{\rm For }i=1,\ldots,m\\
\left\lfloor
\begin{array}{l}
\mathsf{s}_{i,n}=2\mathsf{q}_{n}-\mathsf{p}_{
i,n}+\mathsf{z}_{i,n}-\mathsf{x}_{n}\\
\mathsf{t}_{i,n}=\mathsf{s}_{i,n}-\gamma\mathsf{B}\mathsf{q}_n\\
\mathsf{z}_{i,n+1}=\mathsf{z}_{i,n}+\lambda_n(\mathsf{t}_{i,n}-\mathsf{r}_{i,n}).\\
\end{array}
\right.\\
\end{array}
\right.
\end{align}
Then, $\mathsf{x}_n\weakly\overline{\mathsf{x}}$ for 
some solution $\overline{\mathsf{x}}$ to Problem~\ref{prob:Fad} and $\mathsf{x}_{n+1}-\mathsf{x}_n\to 0$.
\end{theorem}
\begin{proof}
Set, for every $n\in\NN$, $x_n=j(\mathsf{x}_n)$,
$q_n=j(\mathsf{q}_n)$,
$s_n=(\mathsf{s}_{i,n})_{1\leq i\leq m}$,
$z_n=(\mathsf{z}_{i,n})_{1\leq i\leq m}$, and
$p_n=(\mathsf{p}_{i,n})_{1\leq i\leq m}$.
It follows from Proposition~\ref{p:app1}\ref{p:app1i} and
\eqref{e:algo2} that, for every $n\in\NN$, $x_n=P_Vz_n$ and
$q_n=P_Vp_n=P_Vs_n$. Hence,
it follows from \eqref{e:defprodc} and Proposition~\ref{p:app1} that
\eqref{e:algo2} can be written equivalently as
\eqref{e:algo}.
Altogether, Corollary~\ref{c:0} and
Proposition~\ref{p:app1}\ref{p:app1v} yield the results.
\end{proof}

\begin{remark}\
\label{rem:1}
In the particular case when $m=2$, $B=0$, and
$\omega_1=\omega_2=1/2$, the method proposed in Theorem~\ref{t:4}
reduces to 
 \begin{align}
\label{e:algo3}
(\forall n\in\NN)\quad
&\left\lfloor 
\begin{array}{l}
\mathsf{x}_{n}=(\mathsf{z}_{1,n}+\mathsf{z}_{2,n})/2\\
\mathsf{p}_{1,n}=J_{2\gamma
\mathsf{A}_1}(\mathsf{z}_{1,n})\\
\mathsf{p}_{2,n}=J_{2\gamma
\mathsf{A}_2}(\mathsf{z}_{2,n})\\
\mathsf{z}_{1,n+1}=\mathsf{z}_{1,n}+\lambda_n(\mathsf{p}_{2,n}
-\mathsf{x}_{n})\\
\mathsf{z}_{2,n+1}=\mathsf{z}_{2,n}+\lambda_n(\mathsf{p}_{1,n}
-\mathsf{x}_{n}),
\end{array}
\right.
\end{align}
which is exactly the method proposed in \cite[Remark~6.2(ii)]{opti1} for
finding a zero of the sum of two maximally monotone operators  
$\mathsf{A}_1$ and $\mathsf{A}_2$. 
In the case when these resolvents are hard to calculate, \eqref{e:algo3}
provides an alternative method which computes them in parallel.
\end{remark}

\subsection{Primal-Dual Monotone Inclusions}

This section is devoted to the numerical resolution of a very general composite primal-dual monotone inclusion involving vectorial subspaces. A difference of the method in Section~\ref{ssec:41}, the algorithm proposed in this section deals with monotone operators composed with linear transformations and solves simultaneously primal and dual inclusions.

Let us introduce a partial sum of two set-valued operators with respect a closed vectorial subspace. 
This notion is a generalization of the parallel sum (see, e.g., \cite{Bot12} and the references therein).
\begin{definition}
Let $\HH$ be a real Hilbert space, let $U\subset\HH$ be a closed
vectorial subspace, and let $A\colon\HH\to2^{\HH}$ and $B\colon\HH\to
2^{\HH}$ be two non linear operators. The {\em partial sum of
$A$ and $B$ with respect to $U$} is defined by
\begin{equation}
A\infconv_{\!U} B=\big(A_U+B_U\big)_U.
\end{equation}
In particular, we have $A\infconv_{\!\HH} B=A+B$
and $A\infconv_{\!\{0\}} B=A\infconv B=(A^{-1}+B^{-1})^{-1}$.
\end{definition}
Note that, since the operation $A\mapsto A_U$ preserves monotonicity \cite{Spin83}, if $A$ and $B$ are monotone then $A\infconv_{\!U} B$ is monotone as well.
In this section we are interested in the following problem.
\begin{problem}
\label{prob:appPD}
Let $\mathsf{H},(\mathsf{G}_i)_{1\leq i\leq m}$ be real Hilbert spaces,
for every $i\in\{1,\ldots,m\}$, let $\mathsf{U}\subset\mathsf{H}$
and $\mathsf{V}_i\subset\mathsf{G}_i$ be closed vectorial spaces, 
let $\mathsf{A}\colon\mathsf{H}\to 2^{\mathsf{H}}$ and $\mathsf{B}_i\colon\mathsf{G}_i\to 2^{\mathsf{G}_i}$ be 
maximally monotone, let $\mathsf{L}_i\colon\mathsf{H}\to \mathsf{G}_i$ be linear and bounded,
let $\mathsf{D}_i\colon\mathsf{G}_i\to 2^{\mathsf{G}_i}$ be monotone such that
$(\mathsf{D}_i)_{\mathsf{V}_i^{\bot}}$ is $\nu_i$-lipschitzian for some $\nu_i\in\RPP$,
let $\mathsf{C}\colon\mathsf{H}\to\mathsf{H}$ be monotone and $\mu$-lipschitzian for some
$\mu\in\RPP$,
let $\mathsf{z}\in\mathsf{H}$, and let $\mathsf{b}_i\in \mathsf{G}_i$.
The problem is to solve the primal inclusion
\begin{equation}
\label{e:primal}
\text{find}\;\;\mathsf{x}\in\mathsf{H}\quad\text{such that}\quad 
\mathsf{z}\in \mathsf{A}\mathsf{x}+N_\mathsf{U}\mathsf{x}+\sum_{i=1}^m\Big(\mathsf{L}_i^*P_{\mathsf{V}_i}(\mathsf{B}_i\infconv_{\!\mathsf{V}_i^{\bot}}
\mathsf{D}_i+N_{\mathsf{V}_i})P_{\mathsf{V}_i}(\mathsf{L}_i\mathsf{x}-\mathsf{b}_i)\Big)+\mathsf{C}\mathsf{x}
\end{equation}
together with the dual inclusion
\begin{multline}
\label{e:dual}
\text{find}\;\;\mathsf{u}_1\in\mathsf{G}_1,\ldots, \mathsf{u}_m\in\mathsf{G}_m\hspace{.2cm}\text{such
that}\\ 
(\exi \mathsf{x}\in\mathsf{H})\:\:
\begin{cases}
\mathsf{z}-\sum_{i=1}^m\mathsf{L}_i^*P_{\mathsf{V}_i}\mathsf{u}_i\in \mathsf{A}\mathsf{x}+\mathsf{C}\mathsf{x}+N_\mathsf{U}\mathsf{x}\\
(\forall i\in\{1,\ldots,m\})\: \mathsf{u}_i\in P_{\mathsf{V}_i}(\mathsf{B}_i\infconv_{\!\mathsf{V}_i^{\bot}}
\mathsf{D}_i+N_{\mathsf{V}_i})P_{\mathsf{V}_i}(\mathsf{L}_i\mathsf{x}-\mathsf{b}_i).
\end{cases}
\end{multline}

The set of solutions to \eqref{e:primal} is denoted by 
${\mathcal P}$ and the set of solutions to \eqref{e:dual} by 
${\mathcal D}$, which are assumed to be nonempty.
\end{problem}

In the particular case when $\mathsf{U}=\mathsf{H}$ and, for every
$i\in\{1,\ldots,m\}$, $\mathsf{V}_i=\mathsf{G}_i$, Problem~\ref{prob:appPD} reduces
to the problem solved in \cite{Comb12}, where a convergent primal-dual algorithm activating separately each involved operator is proposed. In the case when, for every $i\in\{1,\ldots,m\}$, $\mathsf{V}_i=\mathsf{G}_i$, Problem~\ref{prob:appPD} reduces to the problem addressed in \cite{Bot13}, where a splitting method with ergodic convergence is provided. A disadvantage of this algorithm is the presence of vanishing parameters which may lead to numerical instabilities together with additionally conditions difficult to be verified in general. At the best of our knowledge, the general case has not been tackled in the literature via splitting methods. 

Problem~\ref{prob:appPD} requires a lipschitzian condition on $({\mathsf{D}_i}_{\mathsf{V}_i^{\bot}})_{1\leq i\leq m}$. In the simplest case when, for every $i\in\{1,\ldots,m\}$, $\mathsf{V}_i=\mathsf{G}_i$, this condition reduces to the lipschitzian property on ${\mathsf{D}_i}^{-1}$, which is trivially satisfied, e.g., when $\mathsf{D}_i0=\mathsf{G}_i$ and, for every $y\neq0$, $\mathsf{D}_iy=\varnothing$.
The following proposition furnishes other non-trivial instances in which the partial inverse of a monotone operator with respect to a closed vectorial subspace is lipschitzian.
\begin{proposition}
\label{p:partinvD}
Let $\mathsf{V}$ be a closed vectorial subspace of a real Hilbert space $\mathsf{H}$ and suppose that one of the following holds.
\begin{enumerate}
\item\label{p:partinvDi} $\mathsf{D}\colon\mathsf{H}\to\mathsf{H}$ is $\beta$-strongly monotone and $\nu$-cocoercive.
\item\label{p:partinvDii} $\mathsf{D}=\nabla \mathsf{f}$, where $\mathsf{f}\colon\mathsf{H}\to\RX$ is differentiable, $\beta$-strongly convex, and $\nabla \mathsf{f}$ is $\nu^{-1}$-lipschitzian. 
\item\label{p:partinvDiii} $\mathsf{D}$ is linear bounded operator satisfying, for every $\mathsf{x}\in\mathsf{H}$, $\scal{\mathsf{x}}{\mathsf{D}\mathsf{x}}\geq\beta\|\mathsf{x}\|^2$, and $\nu=\beta/\|\mathsf{D}\|^2$. 
\end{enumerate}
Then $\mathsf{D}_\mathsf{V}$ is $\alpha$-cocoercive and $\alpha$-strongly monotone with $\alpha=\min\{\beta,\nu\}/2$. In particular, $\mathsf{D}_\mathsf{V}$ is $\alpha^{-1}$-lipschitzian.
\end{proposition}
\begin{proof}
\ref{p:partinvDi}:
Let $(\mathsf{x},\mathsf{u})$ and $(\mathsf{y},\mathsf{v})$ in $\gr(\mathsf{D}_\mathsf{V})$. Then it follows from \eqref{e:partialinv} that $(P_\mathsf{V}\mathsf{x}+P_{\mathsf{V}^{\bot}}\mathsf{u},P_\mathsf{V}\mathsf{u}+P_{\mathsf{V}^{\bot}}\mathsf{x})$ and $(P_\mathsf{V}\mathsf{y}+P_{\mathsf{V}^{\bot}}\mathsf{v},P_\mathsf{V}\mathsf{v}+P_{\mathsf{V}^{\bot}}\mathsf{y})$ are in $\gr(\mathsf{D})$, and, from the strong monotonicity assumption on $\mathsf{D}$, we have
\begin{align}
\label{e:auxstrong}
\scal{\mathsf{x}-\mathsf{y}}{\mathsf{u}-\mathsf{v}}&=
\scal{P_\mathsf{V}(\mathsf{x}-\mathsf{y})}{P_\mathsf{V}(\mathsf{u}-\mathsf{v})}+\scal{P_{\mathsf{V}^{\bot}}(\mathsf{u}-\mathsf{v})}{P_{\mathsf{V}^{\bot}}(\mathsf{x}-\mathsf{y})}\nonumber\\
&=\scal{P_\mathsf{V}\mathsf{x}+P_{\mathsf{V}^{\bot}}\mathsf{u}-(P_\mathsf{V}\mathsf{y}+P_{\mathsf{V}^{\bot}}\mathsf{v})}{P_\mathsf{V}\mathsf{u}+P_{\mathsf{V}^{\bot}}\mathsf{x}-(P_\mathsf{V}\mathsf{v}+P_{\mathsf{V}^{\bot}}\mathsf{y})}
\nonumber\\
&\geq
\beta\|P_\mathsf{V}\mathsf{x}+P_{\mathsf{V}^{\bot}}\mathsf{u}-(P_\mathsf{V}\mathsf{y}+P_{\mathsf{V}^{\bot}}\mathsf{v})\|^2\nonumber\\
&=\beta(\|P_\mathsf{V}(\mathsf{x}-\mathsf{y})\|^2+\|P_{\mathsf{V}^{\bot}}(\mathsf{u}-\mathsf{v})\|^2).
\end{align}
Analogously, the cocoercivity assumption on $\mathsf{D}$ yields
$\scal{\mathsf{x}-\mathsf{y}}{\mathsf{u}-\mathsf{v}}
\geq
\nu(\|P_\mathsf{V}(\mathsf{u}-\mathsf{v})\|^2+\|P_{\mathsf{V}^{\bot}}(\mathsf{x}-\mathsf{y})\|^2)$. Hence, it follows from \eqref{e:auxstrong} that
\begin{equation}
\scal{\mathsf{x}-\mathsf{y}}{\mathsf{u}-\mathsf{v}}\geq\frac{\beta}{2}(\|P_\mathsf{V}(\mathsf{x}-\mathsf{y})\|^2+\|P_{\mathsf{V}^{\bot}}(\mathsf{u}-\mathsf{v})\|^2)+\frac{\nu}{2}(\|P_\mathsf{V}(\mathsf{u}-\mathsf{v})\|^2+\|P_{\mathsf{V}^{\bot}}(\mathsf{x}-\mathsf{y})\|^2),
\end{equation}
which yields
$
\scal{\mathsf{x}-\mathsf{y}}{\mathsf{u}-\mathsf{v}}\geq\alpha\big(\|\mathsf{x}-\mathsf{y}\|^2+\|\mathsf{u}-\mathsf{v}\|^2\big)
$
and the result follows.
\ref{p:partinvDii}: From the strong convexity of $\mathsf{f}$ we have that $\mathsf{D}=\nabla \mathsf{f}$ is $\beta$-strongly monotone and it follows from \cite{Bail77} that $\mathsf{D}$ is $\nu$-cocoercive. Hence, the result follows from \ref{p:partinvDi}.
\ref{p:partinvDiii}: Since $\mathsf{D}$ is linear and bounded we have $\|\mathsf{x}\|^2\geq\|\mathsf{D}\mathsf{x}\|^2/\|\mathsf{D}\|^2$. Then $\mathsf{D}$ is $\beta$-strongly monotone and $\nu$-cocoercive and the result follows from \ref{p:partinvDi}.
\end{proof}

The following proposition gives a connection between Problem~\ref{prob:appPD} and Problem~\ref{prob:1}.
\begin{proposition}
\label{p:prodspace}
In the real Hilbert space
$\HH=\mathsf{H}\oplus\mathsf{G}_1\oplus\cdots\oplus\mathsf{G}_m$ set
\begin{equation}
\label{e:defmaxmonprod}
\begin{cases}
A\colon\HH\to 2^{\HH}\colon(\mathsf{x},\mathsf{u}_1,\ldots,\mathsf{u}_m)\mapsto
(-\mathsf{z}+\mathsf{A}\mathsf{x})\times
(P_{\mathsf{V}_1}\mathsf{b}_1+(\mathsf{B}_1)_{\mathsf{V}_1^{\bot}}\mathsf{u}_1)\times\cdots\times(P_{\mathsf{V}_m}\mathsf{b}_m+(\mathsf{B}_m)_{\mathsf{V}_m^{\bot}}\mathsf{u}_m)\\
L\colon\HH\to\HH\colon(\mathsf{x},\mathsf{u}_1,\ldots,
\mathsf{u}_m)\mapsto\big(\sum_ { i=1 }^m\mathsf{L}_i^*P_{\mathsf{V}_i}\mathsf{u}_i,-P_{\mathsf{V}_1}\mathsf{L}_1\mathsf{x},\ldots,-P_{\mathsf{V}_m}\mathsf{L}_m\mathsf{x}\big)\\
C\colon\HH\to\HH\colon(\mathsf{x},\mathsf{u}_1,\ldots,\mathsf{u}_m)\mapsto\big(\mathsf{C}\mathsf{x},
(\mathsf{D}_1)_{\mathsf{V}_1^{\bot}}\mathsf{u}_1,\ldots,(\mathsf{D}_m)_{\mathsf{V}_m^{\bot}}\mathsf{u}_m\big)\\
B\colon\HH\to\HH\colon(\mathsf{x},\mathsf{u}_1,\ldots,\mathsf{u}_m)\mapsto
(C+L)(\mathsf{x},\mathsf{u}_1,\ldots,\mathsf{u}_m)\\
W=\mathsf{U}\times \mathsf{V}_1\times\cdots\times \mathsf{V}_m\\
\chi=\max\{\mu,\nu_1,\ldots,\nu_m\}+\sqrt{\sum_{i=1}^m\|\mathsf{L}_i\|^2}.
\end{cases}
\end{equation}
Then the following hold.
\begin{enumerate}
 \item\label{p:prodspacei} $A$ is maximally monotone and, for every
$\gamma\in\RPP$,
\begin{equation}
J_{\gamma A}\colon(\mathsf{x},\mathsf{u}_1,\ldots,\mathsf{u}_m)\mapsto\Big(J_{\gamma
\mathsf{A}}(\mathsf{x}+\mathsf{z}),J_{\gamma(\mathsf{B}_1)_{\mathsf{V}_1^{\bot}}}(\mathsf{u}_1-P_{\mathsf{V}_1}\mathsf{b}_1),\ldots,
J_{\gamma(\mathsf{B}_m)_{\mathsf{V}_m^{\bot}}
}(\mathsf{u}_m-P_{\mathsf{V}_m}\mathsf{b}_m)\Big).
\end{equation}
\item\label{p:prodspaceii} $L$ is a linear bounded operator, $L^*=-L$, and $\|L\|\leq\sqrt{\sum_{i=1}^m\|\mathsf{L}_i\|^2}$.
\item\label{p:prodspaceiii} $B$ is monotone and $\chi$-lipschitzian.
\item\label{p:prodspaceiv} $W$ is a closed vectorial subspace of $\HH$, $N_{W}\colon(\mathsf{x},\mathsf{u}_1,\ldots,\mathsf{u}_m)\mapsto N_\mathsf{U}\mathsf{x}\times N_{\mathsf{V}_1}\mathsf{u}_1\times\cdots\times N_{\mathsf{V}_m}\mathsf{u}_m$, and
$P_{W}\colon(\mathsf{x},\mathsf{u}_1,\ldots,\mathsf{u}_m)\mapsto(P_\mathsf{U}\mathsf{x},P_{\mathsf{V}_1}\mathsf{u}_1,
\ldots,P_{\mathsf{V}_m}\mathsf{u}_m).$
\item\label{p:prodspacev} 
$\zer(A+B+N_{
W})\subset\mathcal{P}\times\mathcal{D}$.
\item\label{p:prodspacevi} $\mathcal{P}\neq\emp\:\:\Leftrightarrow\:\: \zer(A+B+N_{W})\neq\emp\:\:\Leftrightarrow\:\:\mathcal{D}\neq\emp.$
\end{enumerate}
\end{proposition}
\begin{proof}
\ref{p:prodspacei}: Since, for every $i\in\{1,\ldots,m\}$, $(\mathsf{B}_i)_{\mathsf{V}_i^{\bot}}$ is maximally monotone, the result follows from \cite[Proposition~23.15 and Proposition~23.16]{Livre1}. \ref{p:prodspaceii}: Let us define $\mathsf{M}\colon\mathsf{G}_1\oplus\cdots\oplus\mathsf{G}_m\to\mathsf{H}$ by
$\mathsf{M}\colon(\mathsf{u}_1,\ldots,\mathsf{u}_m)\mapsto\sum_{i=1}^m\mathsf{L}_i^*P_{\mathsf{V}_i}\mathsf{u}_i$. Since $(\mathsf{L}_i)_{1\leq i\leq m}$ and $(P_{\mathsf{V}_i})_{1\leq i\leq m}$ are linear bounded operators, it is easy to check that $\mathsf{M}$ is linear and bounded, $\mathsf{M}^*\colon\mathsf{x}\mapsto(P_{\mathsf{V}_1}\mathsf{L}_1\mathsf{x},\ldots,P_{\mathsf{V}_m}\mathsf{L}_m\mathsf{x})$, and that we can rewrite $L$ as $L\colon(\mathsf{x},\mathsf{u}_1,\ldots,\mathsf{u}_m)\mapsto (\mathsf{M}(\mathsf{u}_1,\ldots,\mathsf{u}_m),-\mathsf{M}^*\mathsf{x})$. Hence, we deduce from \cite[Proposition~2.7(ii)]{Siopt3} that $L$ is linear and bounded, that $L^*=-L$, and that $\|L\|=\|M\|$.
Now, for every $(\mathsf{u}_1,\ldots,\mathsf{u}_m)\in\mathsf{G}_1\oplus\cdots\oplus\mathsf{G}_m$, we have from triangle and H\"older inequalities $\|M(\mathsf{u}_1,\ldots,\mathsf{u}_m)\|\leq\sum_{i=1}^m\|\mathsf{L}_i\|\|P_{\mathsf{V}_i}\|\|\mathsf{u}_i\|\leq \sum_{i=1}^m\|\mathsf{L}_i\|\|\mathsf{u}_i\|\leq\sqrt{\sum_{i=1}^m\|\mathsf{L}_i\|^2}\sqrt{\sum_{i=1}^m\|\mathsf{u}_i\|^2}$, 
which yields the last assertion.

\ref{p:prodspaceiii}: It follows from \ref{p:prodspaceii} that $L$ is linear, bounded, and skew. Therefore, it is monotone and $\|L\|$-lipschitzian. On the other hand, since $\mathsf{C}$ and $(\mathsf{D}_i)_{\mathsf{V}_i^{\bot}}$ are monotone and lipschitzian, $C$ is monotone and $\max\{\mu,\nu_1,\ldots,\nu_m\}$-lipschitzian. Altogether, it follows from \ref{p:prodspaceii} that $B=C+L$ is monotone and $\chi$-lipschitzian. \ref{p:prodspaceiv}: Clear. \ref{p:prodspacev}: Let $(\mathsf{x},\mathsf{u}_1,\ldots,\mathsf{u}_m)\in\mathsf{H}\times\mathsf{G}_1\times\cdots\mathsf{G}_m$. We have from \eqref{e:defmaxmonprod} and Proposition~\ref{p:propiedpi}\ref{p:propiedpiii} that 
\begin{align}
(\mathsf{x},\mathsf{u}_1,\ldots,\mathsf{u}_m)\in\zer(&A+B+N_{W})\nonumber\\
&\Leftrightarrow\quad 
\begin{cases}
\mathsf{0}\in -\mathsf{z}+\mathsf{A}\mathsf{x}+\mathsf{C}\mathsf{x}+\sum_{i=1}^m\mathsf{L}_i^*P_{\mathsf{V}_i}\mathsf{u}_i+N_\mathsf{U}\mathsf{x}\\
\mathsf{0}\in P_{\mathsf{V}_1}\mathsf{b}_1+(\mathsf{B}_1)_{\mathsf{V}_1^{\bot}}\mathsf{u}_1+(\mathsf{D}_1)_{\mathsf{V}_1^{\bot}}\mathsf{u}_1-P_{\mathsf{V}_1}\mathsf{L}_1\mathsf{x}+N_{\mathsf{V}_1}\mathsf{u}_1\\
\:\:\:\:\vdots\\
\mathsf{0}\in P_{\mathsf{V}_m}\mathsf{b}_m+(\mathsf{B}_m)_{\mathsf{V}_m^{\bot}}\mathsf{u}_m+(\mathsf{D}_m)_{\mathsf{V}_m^{\bot}}\mathsf{u}_m-P_{\mathsf{V}_m}\mathsf{L}_m\mathsf{x}+N_{\mathsf{V}_m}\mathsf{u}_m
\end{cases}\nonumber\\
&\Leftrightarrow\quad 
\begin{cases}
\mathsf{0}\in -\mathsf{z}+\mathsf{A}\mathsf{x}+\mathsf{C}\mathsf{x}+\sum_{i=1}^m\mathsf{L}_i^*P_{\mathsf{V}_i}\mathsf{u}_i+N_\mathsf{U}\mathsf{x}\\
P_{\mathsf{V}_1}(\mathsf{L}_1\mathsf{x}-\mathsf{b}_1)\in ((\mathsf{B}_1)_{\mathsf{V}_1^{\bot}}+(\mathsf{D}_1)_{\mathsf{V}_1^{\bot}}+N_{\mathsf{V}_1})\mathsf{u}_1,\: \mathsf{u}_1\in \mathsf{V}_1\\
\:\:\:\:\vdots\\
P_{\mathsf{V}_m}(\mathsf{L}_m\mathsf{x}-\mathsf{b}_m)\in((\mathsf{B}_m)_{\mathsf{V}_m^{\bot}}+(\mathsf{D}_m)_{\mathsf{V}_m^{\bot}}+N_{\mathsf{V}_m})\mathsf{u}_m,\: \mathsf{u}_m\in \mathsf{V}_m
\end{cases}\nonumber\\
&\Leftrightarrow\quad 
\begin{cases}
\mathsf{0}\in -\mathsf{z}+\mathsf{A}\mathsf{x}+\mathsf{C}\mathsf{x}+\sum_{i=1}^m\mathsf{L}_i^*P_{\mathsf{V}_i}\mathsf{u}_i+N_\mathsf{U}\mathsf{x}\\
\mathsf{u}_1\in P_{\mathsf{V}_1}  ((\mathsf{B}_1)_{\mathsf{V}_1^{\bot}}+(\mathsf{D}_1)_{\mathsf{V}_1^{\bot}}+N_{\mathsf{V}_1})^{-1}P_{\mathsf{V}_1}(\mathsf{L}_1\mathsf{x}-\mathsf{b}_1)\\
\:\:\:\:\vdots\\
\mathsf{u}_m\in P_{\mathsf{V}_m} ((\mathsf{B}_m)_{\mathsf{V}_m^{\bot}}+(\mathsf{D}_m)_{\mathsf{V}_m^{\bot}}+N_{\mathsf{V}_m})^{-1}P_{\mathsf{V}_m}(\mathsf{L}_m\mathsf{x}-\mathsf{b}_m)
\end{cases}\nonumber\\
&\label{e:aux2222}\Leftrightarrow\quad 
\begin{cases}
\mathsf{z}-\sum_{i=1}^m\mathsf{L}_i^*P_{\mathsf{V}_i}\mathsf{u}_i\in \mathsf{A}\mathsf{x}+\mathsf{C}\mathsf{x}+N_\mathsf{U}\mathsf{x}\\
\mathsf{u}_1\in P_{\mathsf{V}_1} (\mathsf{B}_1\infconv_{\!\mathsf{V}_1^{\bot}}\mathsf{D}_1+N_{\mathsf{V}_1})P_{\mathsf{V}_1}(\mathsf{L}_1\mathsf{x}-\mathsf{b}_1)\\
\:\:\:\:\vdots\\
\mathsf{u}_m\in P_{\mathsf{V}_m} (\mathsf{B}_m\infconv_{\!\mathsf{V}_m^{\bot}}\mathsf{D}_m+N_{\mathsf{V}_m})P_{\mathsf{V}_m}(\mathsf{L}_m\mathsf{x}-\mathsf{b}_m)
\end{cases}\\
&\Rightarrow \mathsf{z}\in \mathsf{A}\mathsf{x}+N_\mathsf{U}\mathsf{x}+\sum_{i=1}^m\mathsf{L}_i^*P_{\mathsf{V}_i}(\mathsf{B}_i\infconv_{\!\mathsf{V}_i^{\bot}}\mathsf{D}_i+N_{\mathsf{V}_i})P_{\mathsf{V}_i}(\mathsf{L}_i\mathsf{x}-\mathsf{b}_i)+\mathsf{C}\mathsf{x},
\end{align}
which yields $\mathsf{x}\in\mathcal{P}$. Moreover, \eqref{e:aux2222} yields $(\mathsf{u}_1,\ldots,\mathsf{u}_m)\in\mathcal{D}$.

\ref{p:prodspacevi}: We will prove $\mathcal{P}\neq\emp\Rightarrow\mathcal{D}\neq\emp\Rightarrow\zer(A+B+N_W)\neq\emp\Rightarrow\mathcal{P}\neq\emp$. If $\mathsf{x}\in\mathcal{P}$, there exist $(\mathsf{u}_1,\ldots,\mathsf{u}_m)$ such that \eqref{e:aux2222} holds and, hence, $(\mathsf{u}_1,\ldots,\mathsf{u}_m)\in\mathcal{D}$. Now, if $(\mathsf{u}_1,\ldots,\mathsf{u}_m)\in\mathcal{D}$, there exists $\mathsf{x}\in\mathsf{H}$ such that \eqref{e:aux2222} holds and we deduce from the equivalences in \eqref{e:aux2222} that $(\mathsf{x},\mathsf{u}_1,\ldots,\mathsf{u}_m)\in\zer(A+B+N_W)$. The last implication follows from \ref{p:prodspacev}.
\end{proof}

\begin{theorem}
\label{t:PDalgo}
In the setting of Problem~\ref{prob:appPD}, let $\gamma\in\left]0,1/\chi\right[$ where $\chi$ is defined in \eqref{e:defmaxmonprod}, and let $(\lambda_n)_{n\in\NN}$ be
a sequence in $\left[\varepsilon,1\right]$. Moreover, let $\mathsf{x}_0\in\mathsf{H}$, let
$(\mathsf{u}_{i,0})_{1\leq i\leq
m}\in\mathsf{G}_1\times\cdots\times\mathsf{G}_m$,
and iterate, for every $n\in\NN$,
\begin{align}
\label{e:algo4}
&\left\lfloor 
\begin{array}{l}
\mathsf{r}_{1,n}=\mathsf{x}_{n}-\gamma P_{\mathsf{U}}\big(CP_{\mathsf{U}}\mathsf{x}_{n}+\sum_{i=1}^{m}\mathsf{L}_i^*P_{\mathsf{V}_{i}}\mathsf{u}_{i,n}\big)\\
\mathsf{p}_{1,n}=J_{\gamma\mathsf{A}}(\mathsf{r}_{1,n}+\gamma\mathsf{z})\\
\mathsf{s}_{1,n}=2P_{\mathsf{U}}\mathsf{p}_{1,n}-\mathsf{p}_{1,n}+\mathsf{r}_{1,n}-P_{\mathsf{U}}\mathsf{r}_{1,n}\\
\text{\rm For }i=1,\ldots,m\\
\left\lfloor 
\begin{array}{l}
\mathsf{r}_{2,i,n}=\mathsf{u}_{i,n}-\gamma P_{\mathsf{V}_{i}}({\mathsf{D}_i}_{\mathsf{V}_i^{\bot}}P_{\mathsf{V}_{i}}\mathsf{u}_{i,n}-\mathsf{L}_iP_{\mathsf{U}}\mathsf{x}_{n})\\
\mathsf{p}_{2,i,n}=J_{\gamma{\mathsf{B}_i}_{\mathsf{V}_i^{\bot}}}(\mathsf{r}_{2,i,n}-\gamma P_{\mathsf{V}_{i}}\mathsf{b}_i)\\
\mathsf{s}_{2,i,n}=2P_{\mathsf{V}_i}\mathsf{p}_{2,i,n}-\mathsf{p}_{2,i,n}+\mathsf{r}_{2,i,n}-P_{\mathsf{V}_i}\mathsf{r}_{2,i,n}\\
\mathsf{t}_{2,i,n}=\mathsf{s}_{2,i,n}-\gamma P_{\mathsf{V}_{i}}({\mathsf{D}_i}_{\mathsf{V}_i^{\bot}}P_{\mathsf{V}_{i}}\mathsf{s}_{2,i,n}-\mathsf{L}_iP_{\mathsf{U}}\mathsf{s}_{1,n})\\
\mathsf{u}_{i,n+1}=\mathsf{u}_{i,n}+\lambda_n(\mathsf{t}_{2,i,n}-\mathsf{r}_{2,i,n})
\end{array}
\right.\\
\mathsf{t}_{1,n}=\mathsf{s}_{1,n}-\gamma P_{\mathsf{U}}\big(CP_{\mathsf{U}}\mathsf{s}_{1,n}+\sum_{i=1}^{m}\mathsf{L}_i^*P_{\mathsf{V}_{i}}\mathsf{s}_{2,i,n}\big)\\
\mathsf{x}_{n+1}=\mathsf{x}_{n}+\lambda_n(\mathsf{t}_{1,n}-\mathsf{r}_{1,n}).
\end{array}
\right.
\end{align}
Then, $\mathsf{x}_n\weakly\overline{\mathsf{x}}\in\mathsf{H}$ and, for every $i\in\{1,\ldots,m\}$, $\mathsf{u}_{i,n}\weakly\overline{\mathsf{u}}_i\in\mathsf{G}_i$, and  $(P_{\mathsf{U}}\overline{\mathsf{x}},P_{\mathsf{V}_1}\overline{\mathsf{u}}_1,\ldots,P_{\mathsf{V}_m}\overline{\mathsf{u}}_m)$ is a solution to Problem~\ref{prob:appPD}.
Moreover, $\mathsf{x}_{n+1}-\mathsf{x}_n\to 0$ and, for every $i\in\{1,\ldots,m\}$, $\mathsf{u}_{i,n+1}-\mathsf{u}_{i,n}\to 0$.
\end{theorem}
\begin{proof}
For every $n\in\NN$, denote by $z_n=(\mathsf{x}_n,\mathsf{u}_{1,n},\ldots,\mathsf{u}_{m,n})$, $r_n=(\mathsf{r}_{1,n},\mathsf{r}_{2,1,n},\ldots,\mathsf{r}_{2,m,n})$, $p_n=(\mathsf{p}_{1,n},\mathsf{p}_{2,1,n},\ldots,\mathsf{p}_{2,m,n})$, $s_n=(\mathsf{s}_{1,n},\mathsf{s}_{2,1,n},\ldots,\mathsf{s}_{2,m,n})$, and $t_n=(\mathsf{t}_{1,n},\mathsf{t}_{2,1,n},\ldots,\mathsf{t}_{2,m,n})$.
Then, it follows from Proposition~\ref{p:prodspace} that \eqref{e:algo4} is a particular instance of \eqref{e:algo}. Hence, the results follow from Corollary~\ref{c:0} and Proposition~\ref{p:prodspace}\ref{p:prodspacev}.
\end{proof}

\begin{remark}\
\begin{enumerate}
\item Even if Problem~\ref{prob:Fad} can be seen as a particular case of Problem~\ref{prob:appPD}, the methods in \eqref{e:algo4} and \eqref{e:algo3} have different structures. Indeed, in \eqref{e:algo4} dual variables are updated at each iteration, which may be numerically costly in large scale problems, while only primal variables are updated in Theorem~\ref{t:4}. 

\item Algorithm \eqref{e:algo4} activates independently each operator involved in Problem~\ref{prob:appPD}. The algorithm is explicit in each step if the resolvents of $\mathsf{A}$ and $({\mathsf{B}_i}_{\mathsf{V}_i^{\bot}})_{1\leq i\leq m}$ can be computed explicitly. Observe that the resolvent of the partial inverse of a maximally monotone operator can be explicitly found via Proposition~\ref{p:1}\ref{p:1i}.

\item Note that, when $\lambda_n\equiv1$, $\mathsf{U}=\mathsf{H}$, and, for every $i\in\{1,\ldots,m\}$, $\mathsf{V}_i=\mathsf{G}_i$, the method in Theorem~\ref{t:PDalgo} reduces to the algorithm proposed in \cite[Theorem~3.1]{Comb12} with constant step-size 

\item In the simplest case when $m=2$, $\mathsf{z}=\mathsf{A}=\mathsf{C}=\mathsf{b}_1=\mathsf{b}_2=0$, $\mathsf{L}_1=\mathsf{L_2}=\Id$, $\mathsf{U}=\mathsf{H}$, $\mathsf{V}_1\equiv\mathsf{G}_1$, $\mathsf{V}_2\equiv\mathsf{G}_2$, $\mathsf{D}_10=\mathsf{G}_1$, $\mathsf{D}_20=\mathsf{G}_2$, and for every $y\neq0$, $\mathsf{D}_1y=\mathsf{D}_2y=\varnothing$,
we have, for every $i\in\{1,2\}$, ${\mathsf{D}_i}_{V_i^{\bot}}={\mathsf{D}_i}_{\{0\}}=\mathsf{D}_i^{-1}\colon y\mapsto0$ and Problem~\ref{prob:appPD} reduces to find a zero of $\mathsf{B}_1+\mathsf{B}_2$. In this case \eqref{e:algo4} becomes

\begin{align}
\label{e:algo5}
(\forall n\in\NN)\quad
&\left\lfloor 
\begin{array}{l}
\mathsf{p}_{1,n}=J_{\gamma\mathsf{B}_1^{-1}}(\mathsf{u}_{1,n}+\gamma\mathsf{x}_n)\\
\mathsf{p}_{2,n}=J_{\gamma\mathsf{B}_2^{-1}}(\mathsf{u}_{2,n}+\gamma\mathsf{x}_n)\\
\mathsf{x}_{n+1}=\mathsf{x}_n-\gamma\lambda_n(\mathsf{p}_{1,n}+\mathsf{p}_{2,n})\\
\mathsf{u}_{1,n+1}=(1-\lambda_n)\mathsf{u}_{1,n}+\lambda_n\big(\mathsf{p}_{1,n}-\gamma^2(\mathsf{u}_{1,n}+\mathsf{u}_{2,n})\big)\\
\mathsf{u}_{2,n+1}=(1-\lambda_n)\mathsf{u}_{2,n}+\lambda_n\big(\mathsf{p}_{2,n}-\gamma^2(\mathsf{u}_{1,n}+\mathsf{u}_{2,n})\big).
\end{array}
\right.
\end{align}
A difference of the method derived in Remark~\ref{rem:1} for solving this problem, \eqref{e:algo5} update primal and dual variables and solve the primal and dual inclusion, simultaneously.
\end{enumerate}
\end{remark}

\subsection{Zero-Sum Games}
Our last application focus in the problem of finding a Nash equilibrium in continuous zero sum games. Some comments on finite zero-sum games are also provided. This problem can be formulated in the form of Problem~\ref{prob:1} and solved via an algorithm derived from Theorem~\ref{t:1}.
\begin{problem}
\label{prob:selle}
For every $i\in\{1,2\}$, let $\mathsf{H}_i$ and $\mathsf{G}_i$ be real Hilbert spaces, let $\mathsf{C}_i$ be a closed convex subset of $\mathsf{H}_i$, let $\mathsf{L}_i\colon\mathsf{H}_i\to\mathsf{G}_i$ be a linear bounded operator with closed range,
let $\mathsf{S}_i=\menge{\mathsf{x}\in \mathsf{C}_i}{\mathsf{L}_i\mathsf{x}=\mathsf{b}_i}$, where $\mathsf{b}_i=\mathsf{L}_i\mathsf{e}_i$ for some $\mathsf{e}_i\in\mathsf{H}_i$, let $\chi\in\RPP$, and let $\mathsf{f}\colon\mathsf{H}_1\times\mathsf{H}_2\to\RR$ 
be a differentiable function with a $\chi$--lipschitzian gradient
such that, for every $\mathsf{z}_1\in\mathsf{H}_1$, $\mathsf {f}(\mathsf{z}_1,\cdot)$ is concave and, for every $\mathsf{z}_2\in\mathsf{H}_2$, 
$\mathsf{f}(\cdot,\mathsf{z}_2)$ is convex. Moreover suppose that 
$\inte(\mathsf{C}_1-\mathsf{e}_1)\cap\ker\mathsf{L}_1\neq\emp$ and $\inte(\mathsf{C}_2-\mathsf{e}_2)\cap\ker\mathsf{L}_2\neq\emp$.
The problem is to 
\begin{equation}
\label{e:ZS}
\text{find}\quad \mathsf{x}_1\in\mathsf{S}_1\quad \text{and}\quad \mathsf{x}_2\in\mathsf{S}_2\quad
\text{such that}\quad
\begin{cases}
\mathsf{x}_1\in\Argmind{\mathsf{z}_1\in \mathsf{S}_1}{\mathsf{f}(\mathsf{z}_1,\mathsf{x}_2)}\\
\mathsf{x}_2\in\Argmax{\mathsf{z}_2\in\mathsf{S}_2}{\mathsf{f}(\mathsf{x}_1,\mathsf{z}_2)},
\end{cases}
\end{equation}
under the assumption that solutions exist.
\end{problem}

Problem~\ref{prob:selle} is a generic zero-sum game in which the sets $\mathsf{S}_1$ and $\mathsf{S}_2$ are usually convex bounded sets representing mixed strategy spaces. For example, if, for every $i\in\{1,2\}$, $\mathsf{H}_i=\RR^{N_i}$, $\mathsf{C}_i$ is the positive orthant, $\mathsf{G}_i\equiv\RR$, $\mathsf{b}_i\equiv 1$, and $\mathsf{L}_i$ is the sum of the components in the space $\RR^{N_i}$, $\mathsf{S}_i$ is the simplex in $\RR^{N_i}$. In that case, for a bilinear function $\mathsf{f}$, Problem~\ref{prob:selle} reduces to a finite zero-sum game. Beyond this particular case, Problem~\ref{prob:selle} covers continuous zero-sum games in which mixed strategies are distributions and $\mathsf{L}_1$ and $\mathsf{L}_2$ are integral operators.

As far as we know, some attempts for solving \eqref{e:ZS} are proposed in \cite{Atto08,Atto07} in particular cases when the function $\mathsf{f}$ has a special separable structure with specific coupling schemes. In this particular context they propose alternating methods for finding a Nash equilibrium. On the other hand, a method proposed in \cite{Nash} can solve \eqref{e:ZS} when the projections onto $\mathsf{S}_1$ and $\mathsf{S}_2$ are computable. However, in infinite dimension this projections are not always easy to compute, as we will discuss in Example~\ref{ex:1} below. The following result provides an algorithm for solving Problem~\ref{prob:selle} in the general case, which is obtained as a consequence of Corollary~\ref{c:0}. The method avoids the projections onto $\mathsf{S}_1$ and $\mathsf{S}_2$ by alternating simpler projections onto $\mathsf{C}_1$, $\mathsf{C}_2$, $\ker(\mathsf{L}_1)$, and $\ker(\mathsf{L}_2)$. Let us first introduce the {\em generalized Moore-Penrose inverse} of a bounded linear operator $\mathsf{L}\colon\mathsf{H}\to\mathsf{G}$ with closed range, defined by 
$\mathsf{L}^{\dagger}\colon\mathsf{G}\to\mathsf{H}\colon\mathsf{y}\mapsto P_{C_{\mathsf{y}}}0$, where, for every $\mathsf{y}\in\mathsf{G}$, $C_{\mathsf{y}}=\menge{\mathsf{x}\in\mathsf{H}}{\mathsf{L}^*\mathsf{L}\mathsf{x}=\mathsf{L}^*\mathsf{y}}$. The operator $\mathsf{L}^{\dagger}$ is also linear and bounded and, in the particular case when $\mathsf{L}^*\mathsf{L}$ is invertible, $\mathsf{L}^{\dagger}=(\mathsf{L}^*\mathsf{L})^{-1}\mathsf{L}^*$. For further details and properties the reader is referred to \cite[Section~3]{Livre1}.

\begin{theorem}
\label{t:selle}
Under the notation and assumptions of Problem~\ref{prob:selle}, let $\varepsilon\in\left]0,1\right[$, let 
$\gamma\in\left]0,1/\chi\right[$, and let $(\lambda_n)_{n\in\NN}$ be a sequence in $\left[\varepsilon,1\right]$.
Moreover,
let $(\mathsf{z}_{1,0},\mathsf{z}_{2,0})\in\mathsf{H}_1\oplus\mathsf{H}_2$, and iterate, for every $n\in\NN$, 
\begin{align}
\label{e:tseng1selle}
&\left\lfloor 
\begin{array}{l}
\mathsf{u}_{1,n}=\mathsf{z}_{1,n}-\mathsf{L}_1^*\mathsf{L}_1^{*\dagger}\mathsf{z}_{1,n}\\
\mathsf{u}_{2,n}=\mathsf{z}_{2,n}-\mathsf{L}_2^*\mathsf{L}_2^{*\dagger}\mathsf{z}_{2,n}\\
\mathsf{g}_{1,n}=\nabla
\big(\mathsf{f}(\cdot,\mathsf{e}_2+\mathsf{u}_{2,n})\big)(\mathsf{e}_1+\mathsf{u}_{1,n})-\mathsf{L}_1^*\mathsf{L}_1^{*\dagger}\nabla
\big(\mathsf{f}(\cdot,\mathsf{e}_2+\mathsf{u}_{2,n})\big)(\mathsf{e}_1+\mathsf{u}_{1,n})\\
\mathsf{g}_{2,n}=-\nabla
\big(\mathsf{f}(\mathsf{e}_1+\mathsf{u}_{1,n},\cdot)\big)(\mathsf{e}_2+\mathsf{u}_{2,n})+\mathsf{L}_2^*\mathsf{L}_2^{*\dagger}\nabla
\big(\mathsf{f}(\mathsf{e}_1+\mathsf{u}_{1,n},\cdot)\big)(\mathsf{e}_2+\mathsf{u}_{2,n})\\
\mathsf{r}_{1,n}=\mathsf{z}_{1,n}-\gamma\mathsf{g}_{1,n}\\
\mathsf{r}_{2,n}=\mathsf{z}_{2,n}-\gamma\mathsf{g}_{2,n}\\
\mathsf{p}_{1,n}=P_{\mathsf{C}_1}(\mathsf{r}_{1,n}+\mathsf{e}_1)-\mathsf{e}_1\\
\mathsf{p}_{2,n}=P_{\mathsf{C}_2}(\mathsf{r}_{2,n}+\mathsf{e}_2)-\mathsf{e}_2\\
\mathsf{v}_{1,n}=\mathsf{p}_{1,n}-\mathsf{L}_1^*\mathsf{L}_1^{*\dagger}\mathsf{p}_{1,n}\\
\mathsf{v}_{2,n}=\mathsf{p}_{2,n}-\mathsf{L}_2^*\mathsf{L}_2^{*\dagger}\mathsf{p}_{2,n}\\
\mathsf{s}_{1,n}=2\mathsf{v}_{1,n}-\mathsf{p}_{1,n}+\mathsf{L}_1^*\mathsf{L}_1^{*\dagger}\mathsf{r}_{1,n}\\
\mathsf{s}_{2,n}=2\mathsf{v}_{2,n}-\mathsf{p}_{2,n}+\mathsf{L}_2^*\mathsf{L}_2^{*\dagger}\mathsf{r}_{2,n}\\
\mathsf{h}_{1,n}=\nabla
\big(\mathsf{f}(\cdot,\mathsf{e}_2+\mathsf{v}_{2,n})\big)(\mathsf{e}_1+\mathsf{v}_{1,n})-\mathsf{L}_1^*\mathsf{L}_1^{*\dagger}\nabla
\big(\mathsf{f}(\cdot,\mathsf{e}_2+\mathsf{v}_{2,n})\big)(\mathsf{e}_1+\mathsf{v}_{1,n})\\
\mathsf{h}_{2,n}=-\nabla
\big(\mathsf{f}(\mathsf{e}_1+\mathsf{v}_{1,n},\cdot)\big)(\mathsf{e}_2+\mathsf{v}_{2,n})+\mathsf{L}_2^*\mathsf{L}_2^{*\dagger}\nabla
\big(\mathsf{f}(\mathsf{e}_1+\mathsf{v}_{1,n},\cdot)\big)(\mathsf{e}_2+\mathsf{v}_{2,n})\\
\mathsf{t}_{1,n}=\mathsf{s}_{1,n}-\gamma\mathsf{h}_{1,n}
\\
\mathsf{t}_{2,n}=\mathsf{s}_{2,n}-\gamma\mathsf{h}_{2,n}
\\
\mathsf{z}_{1,n+1}=\mathsf{z}_{1,n}+\lambda_n(\mathsf{t}_{1,n}-\mathsf{r}_{1,n})\\
\mathsf{z}_{2,n+1}=\mathsf{z}_{2,n}+\lambda_n(\mathsf{t}_{2,n}-\mathsf{r}_{2,n}).
\end{array}
\right.
\end{align}
Then there exists a solution 
$(\overline{\mathsf{x}}_1,\overline{\mathsf{x}}_2)$ to Problem~\ref{prob:selle} such 
that $\mathsf{z}_{1,n}+\mathsf{e}_1\weakly\overline{\mathsf{x}}_1$ and $\mathsf{z}_{2,n}+\mathsf{e}_2\weakly\overline{\mathsf{x}}_2$.
\end{theorem}
\begin{proof} 
It follows from \cite[Theorem~16.2]{Livre1} that Problem~\ref{prob:selle} can be written equivalently as the problem of finding $\mathsf{x}_1$ and $\mathsf{x}_2$ such that $0\in\partial(\iota_{\mathsf{S}_1}+\mathsf{f}(\cdot,\mathsf{x}_2))(\mathsf{x}_1)$ and $0\in\partial(\iota_{\mathsf{S}_2}-\mathsf{f}(\mathsf{x}_1,\cdot))(\mathsf{x}_2)$,
which, because of \cite[Corollary~16.38]{Livre1}, is equivalent to
\begin{equation}
\label{e:auxselle}
\begin{cases}
0\in N_{\mathsf{S}_1}(\mathsf{x}_1)+\nabla\big(\mathsf{f}(\cdot,\mathsf{x}_2)\big)(\mathsf{x}_1)\\
0\in N_{\mathsf{S}_2}(\mathsf{x}_2)-\nabla\big(\mathsf{f}(\mathsf{x}_1,\cdot)\big)(\mathsf{x}_2).
\end{cases}
\end{equation}
Now since, for every $i\in\{1,2\}$, $\mathsf{S}_i=\mathsf{C}_i\cap\mathsf{L}_i^{-1}(\mathsf{b}_i)=\mathsf{C}_i\cap(\mathsf{e}_i+\ker\mathsf{L}_i)$, it follows from qualification conditions assumed in Problem~\ref{prob:selle} that \eqref{e:auxselle} is equivalent to
\begin{equation}
\label{e:auxselle2}
\begin{cases}
0\in N_{\mathsf{C}_1}(\mathsf{e}_1+\mathsf{z}_1)+N_{\ker\mathsf{L}_1}(\mathsf{z}_1)+\nabla\big(\mathsf{f}(\cdot,\mathsf{e}_2+\mathsf{z}_2)\big)(\mathsf{e}_1+\mathsf{z}_1)\\
0\in N_{\mathsf{C}_2}(\mathsf{e}_2+\mathsf{z}_2)+N_{\ker\mathsf{L}_2}(\mathsf{z}_2)-\nabla\big(\mathsf{f}(\mathsf{e}_1+\mathsf{z}_1,\cdot)\big)(\mathsf{e}_2+\mathsf{z}_2),
\end{cases}
\end{equation}
where $\mathsf{z}_1=\mathsf{x}_1-\mathsf{e}_1$ and $\mathsf{z}_2=\mathsf{x}_2-\mathsf{e}_2$. Hence, by defining
\begin{equation}
\label{e:defselle}
\begin{cases}
V=\ker(\mathsf{L}_1)\times\ker(\mathsf{L}_2)\\
A\colon\mathsf{H}_1\times\mathsf{H}_2\to 2^{\mathsf{H}_1\times\mathsf{H}_2}\colon(\mathsf{z}_1,\mathsf{z}_2)\mapsto N_{\mathsf{C}_1\times\mathsf{C}_2}(\mathsf{e}_1+\mathsf{z}_1,\mathsf{e}_2+\mathsf{z}_2)\\
B\colon\mathsf{H}_1\times\mathsf{H}_2\to\mathsf{H}_1\times\mathsf{H}_2\colon(\mathsf{z}_1,\mathsf{z}_2)\mapsto
\begin{pmatrix}
\nabla\big(\mathsf{f}(\cdot,\mathsf{e}_2+\mathsf{z}_2)\big)(\mathsf{e}_1+\mathsf{z}_1)\\
-\nabla\big(\mathsf{f}(\mathsf{e}_1+\mathsf{z}_1,\cdot)\big)(\mathsf{e}_2+\mathsf{z}_2)
\end{pmatrix},
\end{cases}
\end{equation}
Problem~\ref{prob:selle} is equivalent to find $\mathsf{z}_1\in\mathsf{H}_1$ and $\mathsf{z}_2\in\mathsf{H}_2$ such that 
$0\in A(\mathsf{z}_1,\mathsf{z}_2)+B(\mathsf{z}_1,\mathsf{z}_2)+N_V(\mathsf{z}_1,\mathsf{z}_2)$,
where $V$ is clearly a closed vectorial subspace of $\mathsf{H}_1\times\mathsf{H}_2$, $A$ is maximally monotone \cite[Proposition~20.22]{Livre1}, and $B$ is monotone (\cite[Proposition~20.22]{Livre1} and \cite{Rock70}). Moreover,  since $\nabla\mathsf{f}$ is $\chi$-lipschitzian, $B$ is also $\chi$-lipschitzian. On the other hand, it follows from \cite[Proposition~3.28(iii)]{Livre1} and \cite[Proposition~23.15(iii)]{Livre1} that
$
P_V\colon(\mathsf{z}_1,\mathsf{z}_2)\mapsto\big(\mathsf{z}_1-\mathsf{L}_1^*\mathsf{L}_1^{*\dagger}\mathsf{z}_1,\mathsf{z}_2-\mathsf{L}_2^*\mathsf{L}_2^{*\dagger}\mathsf{z}_2\big)$,
$
J_{\gamma A}\colon(\mathsf{z}_1,\mathsf{z}_2)\mapsto\big(P_{\mathsf{C}_1}(\mathsf{z}_1+\mathsf{e}_1)-\mathsf{e}_1,P_{\mathsf{C}_2}(\mathsf{z}_2+\mathsf{e}_2)-\mathsf{e}_2\big)
$
and we deduce that \eqref{e:tseng1selle} is a particular case of \eqref{e:algo} when $A$, $B$, and $V$ are defined by \eqref{e:defselle}. Altogether, the result follows from Corollary~\ref{c:0}.
\end{proof}

\begin{remark}
Note that the proposed method does not need the projection onto $\mathsf{S}_1$ and $\mathsf{S}_2$ at each iteration, but it converges to solution strategies belonging to these sets. This new feature is very useful in cases in which the projection onto $\mathsf{S}_1$ and $\mathsf{S}_2$ are not available or are not easy to compute as the following example illustrates. 
\end{remark}
\begin{example}
\label{ex:1}
We consider a $2$-player zero-sum game in which $X_1\subset\RR^{N_1}$ is bounded and represents the set of pure strategies of player $1$, and  
$S_1=\menge{f\in L^2(X_1)}
{f\geq 0\:\text{a.e.},\int_{X_1}f(x)dx=1}$
is her set of mixed strategies, which are distributions of probability in $L^2(X_1)$ ($X_2$, $N_2$, and $S_2$ are defined
likewise). We recall that $L^2(X)$ stands for the set of square-integrable functions $f\colon X\subset\RR^n\to\RX$. Moreover, let $F\in L^2(X_1\times X_2)$ be a function representing the payoff for player 1 and let $-F$ be the payoff of player 2. The problem is to
\begin{equation}
\label{e:ZStrong2}
\text{find}\:\: f_1\in S_1\:\:\:\text{and}\:\:\:f_2\in S_2
\:\:\text{such that}\:\:
\begin{cases}
f_1\in\Argmind{g_1\in S_1}{\displaystyle{\int_{X_1}\int_{X_2}F(x_1,x_2)g_1(x_1)f_2(x_2)dx_2dx_1}}\\
f_2\in\Argmax{g_2\in S_2}{\displaystyle{\int_{X_1}\int_{X_2}F(x_1,x_2)f_1(x_1)g_2(x_2)dx_2dx_1}}.
\end{cases}
\end{equation}
Note that $S_1$ and $S_2$ are closed convex sets in $L^2(X_1)$ and $L^2(X_2)$, respectively. Hence, the projection of any square-integrable function onto $S_1$ or $S_2$ is well defined. However, these projections are not easy to compute. A possible way to avoid the explicit computation of these projections is to split $S_1$ and $S_2$ in $S_1=\mathsf{C}_1\cap(\mathsf{e}_1+\ker\mathsf{L}_1)$ and $S_2=\mathsf{C}_2\cap(\mathsf{e}_2+\ker\mathsf{L}_2)$ as in the proof of Theorem~\ref{t:selle}, where, for every $i\in\{1,2\}$, $\mathsf{C}_i=\menge{f\in L^2(X_i)}{f\geq 0\quad\text{a.e.}}$, $\mathsf{e}_i\equiv(m_i(X_i))^{-1}$, $\mathsf{L}_i\colon f\mapsto\int_{X_i}f(x)dx$,
and $m_i(X_i)$ stands for the Lebesgue measure of the set $X_i$. Note that $\mathsf{e}_1\in\inte\mathsf{C}_1\cap(\mathsf{e}_1+\ker\mathsf{L}_1)$ and 
$\mathsf{e}_2\in\inte\mathsf{C}_2\cap(\mathsf{e}_2+\ker\mathsf{L}_2)$, which yield the qualification condition in Problem~\ref{prob:selle}. For every $i\in\{1,2\}$, let $\mathsf{H}_i=L^2(X_i)$ and define the function $\mathsf{f}\colon\mathsf{H}_1\times\mathsf{H}_2\to\RR\colon(f_1,f_2)\mapsto \int_{X_1}\int_{X_2}F(x_1,x_2)f_1(x_1)f_2(x_2)dx_2dx_1$,
which is bilinear, differentiable, and it follows from $F\in L^2(X_1\times X_2)$ that
\begin{equation}
\nabla\mathsf{f}\colon(f_1,f_2)\mapsto\Big(\int_{X_2}F(\cdot,x_2)f_2(x_2)dx_2,\int_{X_1}F(x_1,\cdot)f_1(x_1)dx_1\Big)\in\mathsf{H}_1\times\mathsf{H}_2
\end{equation}
and that $\nabla\mathsf{f}$ is $\chi$-lipschitzian with $\chi=\|F\|_{L^2(X_1\times X_2)}$.
Thus, by defining $\mathsf{G}_i=\RR$, \eqref{e:ZStrong2} is a particular instance of Problem~\ref{prob:selle}. Note that, for every $i\in\{1,2\}$, $\mathsf{L}_i^*\colon\RR\to L^2(X_i)\colon\xi\mapsto\delta_{\xi}$, where, for every $\xi\in\RR$, $\delta_{\xi}\colon x\mapsto \xi$ is the constant function. Moreover, the operator $\mathsf{L}_i\circ \mathsf{L}_i^*\colon\xi\to m_i(X_i)\xi$ is invertible with $(\mathsf{L}_i\circ \mathsf{L}_i^*)^{-1}\colon\xi\mapsto\xi/m_i(X_i)$, which yields
$\mathsf{L}_i^{*\dagger}=(\mathsf{L}_i\circ \mathsf{L}_i^*)^{-1}\mathsf{L}_i=\mathcal{M}_i$, where $ \mathcal{M}_i\colon f\mapsto\delta_{\bar{f}}$ and $\bar{f}=\int_{X_i}f(x)dx/m_i(X_i)$ is the mean value of $f$. In addition, for every $i\in\{1,2\}$,
$P_{\mathsf{C}_i}\colon f\mapsto f_+\colon t\mapsto\max\{0,f(t)\}$.
Altogether, \eqref{e:tseng1selle} reduces to
\begin{align}
\label{e:tseng1selle2}
&\left\lfloor 
\begin{array}{l}
u_{1,n}=z_{1,n}-\mathcal{M}_1(z_{1,n})\\
u_{2,n}=z_{2,n}-\mathcal{M}_2(z_{2,n})\\
g_{1,n}=G_1+\Big(\int_{X_2}F(\cdot,x_2)u_{2,n}(x_2)dx_2-
\mathcal{M}_1\big(\int_{X_2}F(\cdot,x_2)u_{2,n}(x_2)dx_2\big)\Big)\\
g_{2,n}=-G_2-\Big(\int_{X_1}F(x_1,\cdot)u_{1,n}(x_1)dx_1-\mathcal{M}_2\big(\int_{X_1}F(x_1,\cdot)u_{1,n}(x_1)dx_1\big)\Big)\\
r_{1,n}=z_{1,n}-\gamma g_{1,n}\\
r_{2,n}=z_{2,n}-\gamma g_{2,n}\\
p_{1,n}=\big[r_{1,n}+m_1(X_1)^{-1}\big]_+-m_1(X_1)^{-1}\\
p_{2,n}=\big[r_{2,n}+m_2(X_2)^{-1}\big]_+-m_2(X_2)^{-1}\\
v_{1,n}=p_{1,n}-\mathcal{M}_1(p_{1,n})\\
v_{2,n}=p_{2,n}-\mathcal{M}_2(p_{2,n})\\
s_{1,n}=2v_{1,n}-p_{1,n}+\mathcal{M}_1(r_{1,n})\\
s_{2,n}=2v_{2,n}-p_{2,n}+\mathcal{M}_2(r_{2,n})\\
h_{1,n}=G_1+\Big(\int_{X_2}F(\cdot,x_2)v_{2,n}(x_2)dx_2-
\mathcal{M}_1\big(\int_{X_2}F(\cdot,x_2)v_{2,n}(x_2)dx_2\big)\Big)\\
h_{2,n}=-G_2-\Big(\int_{X_1}F(x_1,\cdot)v_{1,n}(x_1)dx_1-\mathcal{M}_2\big(\int_{X_1}F(x_1,\cdot)v_{1,n}(x_1)dx_1\big)\Big)\\
t_{1,n}=s_{1,n}-\gamma h_{1,n}
\\
t_{2,n}=s_{2,n}-\gamma h_{2,n}
\\
z_{1,n+1}=z_{1,n}+\lambda_n(t_{1,n}-r_{1,n})\\
z_{2,n+1}=z_{2,n}+\lambda_n(t_{2,n}-r_{2,n}),
\end{array}
\right.
\end{align}
where 
\begin{equation}
\begin{cases}
G_1\colon z_1\mapsto\displaystyle {\mathcal{M}_2(F(z_1,\cdot))-\frac{1}{m_1(X_1)m_2(X_2)}\int_{X_1}\int_{X_2}F(x_1,x_2)dx_2dx_1}\\
G_2\colon z_2\mapsto\displaystyle {\mathcal{M}_1(F(\cdot,z_2))-\frac{1}{m_1(X_1)m_2(X_2)}\int_{X_1}\int_{X_2}F(x_1,x_2)dx_2dx_1
}\end{cases}
\end{equation}
and $\gamma\in]0,1/\chi[$. Altogether, Theorem~\ref{t:selle} asserts that the sequences $(z_{1,n}+m_1(X_1)^{-1})_{n\in\NN}$ and $(z_{2,n}+m_2(X_2)^{-1})_{n\in\NN}$ 
converge to $\overline{f}_1\in\mathsf{H}_1$ and
$\overline{f}_2\in\mathsf{H}_2$, respectively, where 
$(\overline{f}_1,\overline{f}_2)$ is a solution to 
\eqref{e:ZStrong2}. 

Note that, in the particular case when $X_1$ and $X_2$ are finite sets of actions (or pure strategies), $S_1$ and $S_2$ are finite dimensional simplexes, and
$F\colon (x_1,x_2)\mapsto x_1^{\top}\mathsf{F}x_2$ is a payoff matrix. In this case \eqref{e:tseng1selle2} provides a first order method for finding Nash equilibria in the finite zero-sum game  (for complements and background on finite games, see~\cite{Weib95})
\begin{equation}
\label{e:ZS3}
\text{find}\quad x_1\in S_1\:\:\:\text{and}\:\:\:x_2\in S_2
\quad\text{such that}\quad
\begin{cases}
x_1\in\Argmind{y_1\in S_1}{x_1^{\top}\mathsf{F}x_2}\\
x_2\in\Argmax{y_2\in S_2}{x_1^{\top}\mathsf{F}x_2}.
\end{cases}
\end{equation}
When a large number of pure actions are involved (e.g., Texas Hold'em poker) classical linear programming methods for solving \eqref{e:ZStrong2} are enormous and unsolvable via standard algorithms as simplex. Other attempts using acceleration schemes for obtaining good convergence rates are provided in \cite{Gilp12}.  However, the proposed method does not guarantee the convergence of the iterates. Other methods need to compute a Nash equilibrium at each iteration, which is costly numerically \cite{Zink99}. The method obtained from \eqref{e:tseng1selle2} is an explicit convergent method that solves \eqref{e:ZS3} overcoming previous difficulties. Numerical simulations and comparisons with other methods in the literature are part of further research.
\end{example}
\section{Conclusions}
We provide a fully split algorithm for finding a zero of $A+B+N_V$. The proposed method exploits the intrinsic properties of each of the operators involved by activating explicitly the single-valued operator $B$ and by computing the resolvent of $A$ and projections onto $V$. Weak convergence to a zero of $A+B+N_V$ is guaranteed and applications to monotone inclusions involving $m$ maximally monotone operators, to primal-dual composite inclusions involving partial sums of monotone operators, and continuous zero-sum games are provided. In addition, the partial sum of two set-valued operators with respect to a closed vectorial subspace is introduced. This operation preserves monotonicity and a further study will be done in a future work. Furthermore, in the zero-sum games context, a splitting method is provided for computing Nash equilibria. The algorithm replaces the projections onto mixed strategy spaces (infinite dimensional simplexes) by alternate simpler projections.

\end{document}